\documentclass[reqno]{amsart}

\usepackage[utf8]{inputenc}
\usepackage{amssymb}
\usepackage{graphicx}
\usepackage{latexsym}
\usepackage{stmaryrd}
\usepackage{enumitem}
\usepackage[bookmarks]{hyperref}
\usepackage{xcolor}

\usepackage{geometry}
\usepackage{hyperref}
\hypersetup{
pdftitle = {Tree-of-tangles for infinite tangles},
pdfsubject = {Tree-of-tangles for infinite tangles},
pdfauthor = {Ann-Kathrin Elm and Jan Kurkofka},
pdfkeywords = {},
draft = false,
bookmarksopen=true}
\hypersetup{
    colorlinks,
    linkcolor={red!60!black},
    citecolor={green!60!black},
    urlcolor={blue!60!black}
}
\usepackage{tikz}
\usetikzlibrary{calc,through,intersections,arrows, trees, positioning, decorations.pathmorphing, cd,decorations.pathreplacing,arrows.meta}
\colorlet{darkishRed}{red!60!black}
\colorlet{darkishBlue}{blue!60!black}
\definecolor{pictureorange}{rgb}{1,0.647,0}
\definecolor{picturedarkred}{rgb}{0.545,0,0}
\colorlet{picturegreen}{green!60!black}
\definecolor{colourdot}{rgb:hsb}{0.85,0.8,1}
\colorlet{colourray}{picturedarkred}
\colorlet{coloura}{picturegreen}
\colorlet{colourxi}{blue}
\pgfdeclarelayer{filllayer}
\pgfsetlayers{filllayer,main}

\usepackage{relsize}
\usepackage{comment}
\usepackage{mathrsfs}
\usepackage{svg}
\usepackage{mathtools}

\newcommand{\modG}{\vert G\vert_{\Theta}}
\newcommand{\minG}{\vert G\vert_{\Gamma}}

\newcommand{\pr}{\normalfont\text{pr}}

\newcommand{\crit}{\normalfont\text{crit}}
\newcommand{\torso}{\normalfont\text{torso}}
\newcommand{\CX}{\breve{\cC}_X}
\newcommand{\CY}{\breve{\cC}_Y}
\newcommand{\CZ}{\breve{\cC}_Z}
\newcommand{\CC}[1]{\breve{\cC}_{#1}}
\newcommand{\invLim}{\varprojlim}

\newcommand{\rest}{\upharpoonright}

\newcommand{\St}{S_{\normalfont\text{t}}}
\newcommand{\Tt}{\Theta_{\normalfont\text{t}}}
\newcommand{\cb}{c}

\newcommand{\rsep}[2]{({#1},{#2})}
\newcommand{\lsep}[2]{({#1},{#2})}
\newcommand{\sep}[2]{\{{#1},{#2}\}}

\newcommand{\crsep}[1]{\rsep{#1}{\breve{\cC}_{#1}}}

\newcommand{\csep}[1]{\sep{#1}{\breve{\cC}_{#1}}}

\newcommand{\trsep}[1]{\rsep{#1}{\ccK({#1})}}
\newcommand{\tlsep}[1]{\lsep{\ccK({#1})}{#1}}
\newcommand{\tsep}[1]{\sep{#1}{\ccK({#1})}}

\DeclareMathOperator{\medcup}{\mathsmaller{\bigcup}}

\newcommand{\bdot}{\boldsymbol{.}\,}

\renewcommand{\subset}{\subseteq}
\renewcommand{\supset}{\supseteq}

\def\comp{com\-pac\-ti\-fi\-ca\-tion}

\def\HD{Haus\-dorff}
\def\HDcomp{\HD\ \comp}
\def\SC{Stone-Čech}

\def\principal{principal}
\def\admissable{admissable}

\newcommand{ \N } { \mathbb{N} }

\newcommand{\closure}[1]{\overline{#1}}
\newcommand{\closureIn}[2]{\closure{#1}^{#2}}
\newcommand{\closureInExt}[2]{\normalfont\text{cl}\,{}_{#2}\,({#1})}

\makeatletter

\def\calCommandfactory#1{%
   \expandafter\def\csname c#1\endcsname{\mathcal{#1}}}
\def\frakCommandfactory#1{%
   \expandafter\def\csname frak#1\endcsname{\mathfrak{#1}}}

\newcounter{ctr}
\loop
  \stepcounter{ctr}
  \edef\X{\@Alph\c@ctr}
  \expandafter\calCommandfactory\X
  \expandafter\frakCommandfactory\X
  \edef\Y{\@alph\c@ctr}
  \expandafter\frakCommandfactory\Y
\ifnum\thectr<26
\repeat

\renewcommand{\cC}{\mathscr{C}}
\renewcommand{\cD}{\mathscr{D}}
\renewcommand{\cP}{\mathscr{P}}
\newcommand{\ccK}{\mathscr{K}}

\newcommand{\dc}[1]{\lceil #1\rceil}
\newcommand{\uc}[1]{\lfloor #1\rfloor}

\newtheorem{theorem}{Theorem}[section]
\newtheorem*{theoremNN}{Theorem}
\newtheorem{proposition}[theorem]{Proposition}
\newtheorem{corollary}[theorem]{Corollary}
\newtheorem{lemma}[theorem]{Lemma}

\newtheorem{mainresult}{Theorem}

\newtheorem{rem}[theorem]{Remark}

\newtheorem{innercustomthm}{Theorem}

\newenvironment{customthm}[1]
  {\renewcommand\theinnercustomthm{#1}\innercustomthm}
  {\endinnercustomthm}

\theoremstyle{definition}
\newtheorem{example}[theorem]{Example}

\newtheorem{definition}[theorem]{Definition}

\theoremstyle{remark}

\newtheorem*{notation}{Notation}

\newtheorem*{ack}{Acknowledgement}

\setenumerate{label={\normalfont (\roman*)},itemsep=0pt}

\def\lowfwd #1#2#3{{\mathop{\kern0pt #1}\limits^{\kern#2pt\raise.#3ex
\vbox to 0pt{\hbox{$\scriptscriptstyle\rightarrow$}\vss}}}}
\def\lowbkwd #1#2#3{{\mathop{\kern0pt #1}\limits^{\kern#2pt\raise.#3ex
\vbox to 0pt{\hbox{$\scriptscriptstyle\leftarrow$}\vss}}}}
\def\fwd #1#2{{\lowfwd{#1}{#2}{15}}}
\def\vS{{\hskip-1pt{\fwd S3}\hskip-1pt}}
\def\vSdash{{\mathop{\kern0pt S\lower-1pt\hbox{${}
     \scriptstyle'$}}\limits^{\kern2pt\raise.1ex
     \vbox to 0pt{\hbox{$\scriptscriptstyle\rightarrow$}\vss}}}}
\def\Sinf{S_{\aleph_0}}
\def\vSinf{\vS_{\mkern-.85\thinmuskip\aleph_0}}
\def\vSt{\vS_{\mkern-.85\thinmuskip\mathrm{t}}}
\def\vE{{\hskip-1pt{\fwd{E}{3.5}}\hskip-1pt}}
\def\vT{\lowfwd T{0.3}1}
\def\vN{\lowfwd N{1}1}
\def\vcB{\lowfwd \cB{1.5}1}
\def\vF{{\hskip-1pt{\fwd{F}{3.5}}\hskip-1pt}}
\def\sv{\lowbkwd s{0}1}
\def\vs{\lowfwd s{1.5}1}
\def\sv{\lowbkwd s{0}1}
\def\vr{\lowfwd r{1.5}2}
\def\rv{\lowbkwd r02}
\def\vSd{{\mathop{\kern0pt S\lower-1pt\hbox{${}
     \scriptstyle'$}}\limits^{\kern2pt\raise.1ex
     \vbox to 0pt{\hbox{$\scriptscriptstyle\rightarrow$}\vss}}}}

\newenvironment{onlayer}[1]
	{\begin{scope}\begin{pgfonlayer}{#1}}
	{\end{pgfonlayer}\end{scope}}

\begin{document}

\title[A tree-of-tangles theorem for infinite tangles]{A tree-of-tangles theorem for
infinite tangles}

\author{Ann-Kathrin Elm}
\author{Jan Kurkofka}
\address{Universität Hamburg, Department of Mathematics, Bundesstraße 55 (Geomatikum), 20146 Hamburg, Germany}
\email{ann-kathrin.elm@uni-hamburg.de, jan.kurkofka@uni-hamburg.de}

\keywords{infinite graph; tree of tangles; aleph 0 tangle; infinite order tangle; end; ultrafilter tangle; nested; tree set; tree decomposition; critical vertex set; distinguish; efficiently; collectionwise normal}

\subjclass[2010]{05C63, 05C83, 05C05, 05C40, 54D15}

\begin{abstract}
Carmesin has extended Robertson and Seymour's tree-of-tangles theorem to the infinite tangles of locally finite infinite graphs.
We extend it further to the infinite tangles of all infinite graphs.

Our result has a number of applications for the topology of infinite graphs, such as their end spaces and their \comp s.

\end{abstract}
\vspace*{-1.14cm}
\maketitle

\vspace*{-.75cm}

\section{Introduction}

\noindent The \emph{tree-of-tangles theorem}, one of the cornerstones of Robertson and Seymour's proof of their graph-minor theorem, says (in the terminology of~\cite[§12.5]{BibleNew}): 

\begin{theoremNN}
Every finite graph $G$ has a nested set of separations which efficiently distinguishes all the finite tangles in $G$ that can be distinguished.
\end{theoremNN}
\noindent This is Theorem~12.5.4 in~\cite{BibleNew}, the original article is~\cite{GMX}. 

Recently, Carmesin~\cite{carmesin2014all} has extended the tree-of-tangles theorem to the infinite tangles of infinite graphs that are locally finite.
The precise statement of Carmesin's result reads:
\begin{theoremNN}
Every infinite connected graph $G$ has a nested set of separations which efficiently distinguishes all the ends of $G$.
\end{theoremNN}
\noindent Note that, in the wording of his theorem, Carmesin does not require the graph to be locally finite, and he speaks of ends where one expects infinite tangles.
This is because his result is more general than an extension of the tree-of-tangles theorem to the infinite tangles of locally finite infinite graphs.
To understand the difference, let us look at how the ends of a graph are related to its infinite tangles.

An end $\omega$ of a graph $G$ (see~\cite{BibleNew}) orients every finite-order separation $\{A,B\}$ of $G$ towards the side that contains a tail from every ray in $\omega$.
Since these orientations are, for distinct separations, consistent in a number of ways, they form an infinite tangle of $G$.
Conversely, every infinite tangle of a locally finite and connected graph $G$ is defined by an end in this way~\cite{EndsAndTangles,Ends}.
Thus, if $G$ is locally finite and connected, there is a canonical bijection between its infinite tangles and its ends.
In this way, Carmesin's result extends the tree-of-tangles theorem to the infinite tangles of locally finite graphs.

When $G$ is not locally finite, however, there can be infinite tangles that are not defined by an end. Then Carmesin's result no longer extends the tree-of-tangles theorem to the infinite tangles of~$G$.

The infinite tangles that do not come from ends of the graph are fundamentally different from ends.
They are closely related to free ultrafilters, and are called \emph{ultrafilter tangles}~\cite{EndsAndTangles}.
More explicitly, by a recent result from~\cite{EndsTanglesCrit}, 
there is a canonical bijection between the ultrafilter tangles and the \emph{ultrafilter tangle blueprints}: pairs $(X,U)$ of a critical vertex set $X$ and a free ultrafilter $U$ on $\CX$, where a finite set $X\subset V(G)$ is \emph{critical} if the collection $\CX$ of the components of $G-X$ whose neighbourhood is equal to~$X$ is infinite.
Therefore, every ultrafilter tangle $\tau=(X,U)$ has two aspects:
Its combinatorial aspect is captured by its blueprint's critical vertex set $X$, and its ultrafilter aspect is encoded by the free ultrafilter~$U$ (see Section~\ref{PropertiesOfInfOrderTangles} for details).
Since every vertex in a critical vertex set has infinite degree, it follows that locally finite connected graphs have no ultrafilter tangles, so all their infinite tangles are ends.

Ultrafilter tangles are interesting also for topological reasons.
Every locally finite connected graph can be naturally compactified by its ends to form its well known end \comp ~\cite{BibleNew} introduced by Freudenthal~\cite{freudenthal}.
But for a non-locally finite graph, adding its ends no longer suffices to compactify it.
Adding its ends \emph{plus its ultrafilter tangles}, however, (i.e.\ adding all its infinite tangles) does again compactify the graph.
This is Diestel's \emph{tangle compactification}~\cite{EndsAndTangles}.
The tangle \comp\ generalises the end \comp\ twofold.
On the one hand, it defaults to the end \comp\ when the graph is locally finite and connected.
And on the other hand, the relation between the end \comp\ of locally finite connected graphs and their \SC\ \comp\ extends to all graphs when ends are generalised to tangles~\cite{StoneCechTangles}.

As our main result, we extend Robertson and Seymour's tree-of-tangles theorem to the infinite tangles of infinite graphs (and thus, we extend Carmesin's result from ends to all infinite tangles):
\begin{mainresult}\label{TreeSetForInfTangles}
Every infinite connected graph $G$ has a nested set of finite-order separations that efficiently distinguishes all the combinatorially distinguishable infinite tangles of~$G$.
\end{mainresult}

\noindent Here, two ultrafilter tangles are \emph{combinatorially distinguishable} if their critical vertex sets are distinct.
The proof will show that, conversely, the nested set of separations we find does not distinguish any two infinite tangles that are combinatorially indistinguishable.
As we will show, our result is best possible in the following sense. 
If a graph $G$ has an ultrafilter tangle~$\tau$, then no nested set of finite-order separations of $G$ efficiently distinguishes all the ultrafilter tangles of $G$ that are not combinatorially distinguishable from~$\tau$.

\subsection*{Applications}

Our work has four applications.

Elbracht, Kneip and Teegen need it in their paper~\cite{InfiniteSplinters}.
So do Bürger and the second author~\cite{StarComb4TheUndominatingStar}.

Our third application is the following structural connectivity result for infinite graphs, which generalises the way in which the cutvertices of a graph decompose it into its blocks in a tree-like fashion.
Call a graph \emph{tough} if deleting finitely many vertices from it never leaves more than finitely many components.
By the pigeonhole principle a graph is tough if and only if it has no critical vertex set.

\begin{mainresult}\label{toughCriticalTreeSet}
Every connected graph $G$ has a nested set of separations whose separators are precisely the critical vertex sets of $G$ and all whose torsos are tough.
\end{mainresult}

(See Section~\ref{subsec:TreeSets} for definitions.)

Theorem~\ref{toughCriticalTreeSet} is interesting also from the perspective of topological infinite graph theory, in view of the following two results.
Diestel and Kühn~\cite{VTopComp} showed that a graph is compactified by its ends if and only if it is tough (i.e., if and only if it has no critical vertex sets), and in~\cite{EndsTanglesCrit} it was shown that every graph is compactified by its ends plus critical vertex sets.
So a graph is compactified by points that come in two types, ends and critical vertex sets, and the second type decomposes the graph into a nested set of separations all whose torsos are compactified by the points of the first type.

Our fourth application answers a question that arises from the work of Polat and of Sprüssel.
End spaces of graphs, in general, are not compact.
However, Polat~\cite{PolatEME1} and Sprüssel~\cite{NormalEnd} independently showed that end spaces of graphs are normal.
Polat even showed that end spaces of graphs are collectionwise normal, which is stronger than normal but weaker than compact \HD . (In a \emph{collectionwise} normal space one can at once pairwise separate any \emph{collection} of closed disjoint sets with disjoint open neighbourhoods, cf.~Definition~\ref{DefinitionCollectionwiseNormal}.)

The infinite tangle space, endowed with the subspace topology of the tangle \comp , contains the end space as a subspace.
As Diestel~\cite{EndsAndTangles} showed, the infinite tangle space is compact \HD , which implies collectionwise normality by general topology.

The ultrafilter tangle space, endowed with the subspace topology of the infinite tangle space, is not usually compact.
Since the infinite tangle space is the disjoint union of the end space and the ultrafilter tangle space, the question arises whether the ultrafilter tangle space is collectionwise normal as well.
We~answer this question in the affirmative:

\begin{mainresult}\label{CollectionWiseNormality}
The ultrafilter tangle space of a graph is collectionwise normal.
\end{mainresult}

Our paper is organised as follows.
Background knowledge is supplied in Section~\ref{sec:preleminaries}.
In Section~\ref{sec:examples} we study examples and show that our main result is best possible.
In Section~\ref{sec:strategy} we give an overview on our overall proof strategy.
Our main technical results are stated and proved in Section~\ref{sec:PrincipalVertexSetsToTreeSets}.
In Section~\ref{sec:Applications} we provide the applications of our main technical results.
In Section~\ref{sec:graphfromtorso} we introduce an equivalence relation on a tree set given a consistent orientation of that tree set. This is the foundation for the definition of the modified torsos and proxies as well as for a `lifting' process that we need in Section~\ref{sec:MainProof}.
In Section~\ref{sec:MainProof}, finally, we introduce the modified torsos and prove our main result.
Section~\ref{sec:appendix} is our appendix.

Throughout this paper, $G=(V,E)$ is a connected graph of arbitrary cardinality.

\begin{ack}
We are grateful to Nathan Bowler, Christian Elbracht, Konstantinos Stavropoulos and Maximilian Teegen for stimulating discussions that contributed to the genesis of this paper.
\end{ack}

\section{Tools and terminology}\label{sec:preleminaries}

\noindent We use the notation of Diestel's book~\cite{BibleNew}.
For a short reminder on compactifications and inverse limits see the appendix Section~\ref{sec:appendix}.


\subsection{Ends of graphs, and inverse limits}

We write $\cX=\cX(G)$ for the collection of all finite subsets of the vertex set $V$ of $G$, partially ordered by inclusion.
An \emph{end} of $G$, as defined by Halin~\cite{halin64}, is an equivalence class of rays of $G$, where a ray is a one-way infinite path.
Here, two rays are said to be \emph{equivalent} if for every $X\in\cX$ both have a subray (also called \emph{tail}) in the same component of $G-X$. 
So in particular every end $\omega$ of $G$ chooses, for every $X\in\cX$, a unique component $C(X,\omega)$ of $G-X$ in which every ray of $\omega$ has a tail. 
In this situation, the end $\omega$ is said to \emph{live} in $C(X,\omega)$.
The set of ends of a graph $G$ is denoted by $\Omega=\Omega(G)$.
If $\cC$ is any collection of components of $G-X$ for some $X\in\cX$, we write $\Omega(X,\cC)$ for the set of ends $\omega$ of $G$ with $C(X,\omega)\in\cC$.
The sets $\Omega(X,\cC)$ form a basis for a topology on~$\Omega$.

Recall that a \emph{comb} is the union of a ray $R$ (the comb's \emph{spine}) with infinitely many disjoint finite paths, possibly trivial, that have precisely their first vertex on~$R$.
The last vertices of those paths are the \emph{teeth} of this comb.
Let us say that an end $\omega$ of $G$ is contained \emph{in the closure} of $M$, where $M$ is either  a subgraph of $G$ or a set of vertices of $G$, if for every $X\in\cX$ the component $C(X,\omega)$ meets $M$.
Equivalently, $\omega$ lies in the closure of $M$ if and only if $G$ contains a comb with all its teeth in $M$ and with its spine in $\omega$.
See~\cite{StarComb1StarsAndCombs,StarComb2TheDominatedComb,StarComb3TheUndominatedComb,StarComb4TheUndominatingStar} for more on combs.


Now we describe an inverse system giving the end space:
We note that $\cX$ is directed by inclusion, and for every $X\in\cX$ we let $\cC_X$ be the set of components of $G-X$.
Then letting $\cb_{X',X}\colon \cC_{X'}\to\cC_X$ for $X'\supseteq X$ send each component of $G-X'$ to the unique component of $G-X$ including it turns the sets $\cC_X$ into an inverse system $\{\cC_X,\cb_{X',X},\cX\}$. 
Clearly, its inverse limit consists precisely of the \emph{directions} of the graph: choice maps $f$ assigning to every $X\in\cX$ a component of $G-X$ such that $f(X')\subseteq f(X)$ whenever $X'\supseteq X$.
Every end $\omega$ of $G$ defines a unique direction~$f_\omega$ by mapping every $X\in\cX$ to~$C(X,\omega)$.
In 2010, Diestel and Kühn~\cite{Ends} showed that, conversely, every direction in fact comes from a unique end in this way:

\begin{theorem}[{\cite[Theorem 2.2]{Ends}}]\label{EndsAreDirections}
Let $G$ be any graph. Then the map $\omega\mapsto f_\omega$ is a homeomorphism between the ends of $G$ and its directions, i.e.~$\Omega\cong\invLim{}\cC_X$.
\end{theorem}

\subsection{Separations of sets and abstract separation systems}

Separation systems are a fundamental notion in graph minor theory.
In this section, we briefly recall the definitions from~\cite{BibleNew,AbstractSepSys,RhdTreeSets} that we need, without detailed explanations: for these we refer to the citations.

A \emph{separation of a set} $V$ is an unordered pair $\{A,B\}$ such that $A\cup B=V$.
The ordered pairs $(A,B)$ and $(B,A)$ are its \emph{orientations}.
Then the \emph{oriented separations} of $V$ are the orientations of its separations.
The map that sends every oriented separation $(A,B)$ to its \emph{inverse} $(B,A)$ is an involution that reverses the partial ordering
\begin{align*}
    (A,B)\le (C,D)\;:\Leftrightarrow\;A\subset C\text{ and }B\supset D
\end{align*}
since $(A,B)\le (C,D)$ is equivalent to $(D,C)\le (B,A)$.

More generally, a \emph{separation system} is a triple $(\vS,{\le},{}^\ast)$ where $(\vS,{\le})$ is a partially ordered set and ${}^\ast\colon\vS\to\vS$ is an order-reversing involution.
We refer to the elements of $\vS$ as \emph{oriented separations}.
If an oriented separation is denoted by $\vs$, then we denote its \emph{inverse} $\vs^\ast$ as $\sv$, and vice versa.
That ${}^\ast$ is \emph{order-reversing} means $\vr\le\vs\leftrightarrow\rv\ge\sv$ for all $\vr,\vs\in\vS$.

A \emph{separation} is an unordered pair of the form $\{\!\vs,\sv\}$, and then denoted by $s$.
Its elements $\vs$ and $\sv$ are the \emph{orientations} of $s$.
The set of all separations $\{\!\vs,\sv\}\subset\vS$ is denoted by $S$.
When a separation is introduced as $s$ without specifying its elements first, we use $\vs$ and $\sv$ (arbitrarily) to refer to these elements.
Every subset $S'\subset S$ defines a separation system $\vSd:=\bigcup S'\subset\vS$ with the ordering and involution induced by~$\vS$.

Separations of sets, and their orientations, are an instance of this abstract setup if we identify $\{A,B\}$ with $\{\,(A,B)\,,(B,A)\,\}$.
Here is another example:
The set $\vE(T):=\{\,(x,y)\mid xy\in E(T)\,\}$ of all \emph{orientations} $(x,y)$ of the edges $xy=\{x,y\}$ of a tree $T$ forms a separation system with the involution $(x,y)\mapsto (y,x)$ and the natural partial ordering on $\vE(T)$ in which $(x,y)<(u,v)$ if and only if $xy\neq uv$ and the unique $\{x,y\}$--$\{u,v\}$ path in $T$ links $y$ to~$u$. 

In the context of a given separation system $(\vS,{\le},{}^\ast)$, a \emph{star (of separations)} is a subset $\sigma\subset\vS$ such that $\vr\le\sv$ for all distinct $\vr,\vs\in\sigma$; see Figure~\ref{fig:star} for an illustration.\footnote{Officially, in~\cite{AbstractSepSys} a star $\sigma$ is additionally required to consist only of oriented separations $\vs$ satisfying $\vs\neq\sv$. In this paper, however, all separations considered will satisfy this condition, which is why we will hide it for the convenience of the reader.}
If $t$ is a node of a tree $T$, then the set
\begin{align*}
    \vF_{\! t}:=\{\,(x,t)\mid xt\in E(T)\,\}
\end{align*}
is a star in $\vE(T)$.
\begin{figure}[t]
\centering
\includegraphics[scale=1]{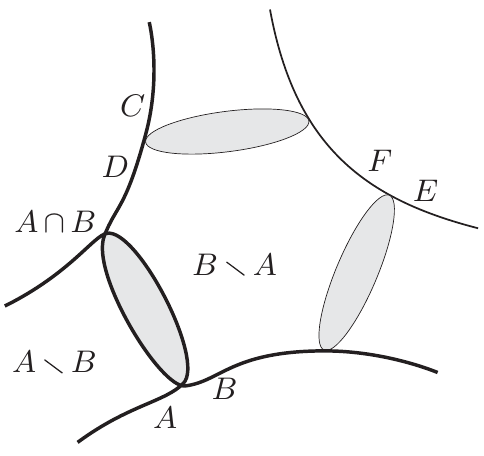}
\caption{The separations $(A,B)$, $(C,D)$, $(E,F)$ form a star~\cite{BibleNew}}
\label{fig:star}
\end{figure}

\subsection{Orientations}

An \emph{orientation} of a separation system $\vS$, or of a set $S$ of separations, is a subset $O\subset\vS$ such that $\vert\, O\cap\{\vs,\sv\}\,\vert=1$ for all~$s\in S$.
A \emph{partial orientation} of $S$ is an orientation of a subset of~$S$.
A subset $O\subset\vS$ is \emph{consistent} if there are no two distinct separations $r,s\in S$ with orientations $\vr<\vs$ and $\rv,\vs\in O$.
For example, the down-closure of any star $\vF_{\! t}$ in $\vE(T)$ is a consistent orientation.




\subsection{Nested sets of separations and tree sets}

Two separations are \emph{nested} if they have comparable orientations.
Two oriented separations $\vr,\vs$ are \emph{nested} if $r$ and $s$ are nested.
A set, either of separations or of oriented separations, is \emph{nested} if every two of its elements are nested.
For example, if $T$ is a tree, then both $E(T)$ and $\vE(T)$ are nested.

To state the definition of a tree set, we need the following definitions.
An oriented separation $\vr\in\vS$ is
\begin{enumerate}
\item \emph{degenerate} if $\vr=\rv$,
\item \emph{trivial} if there is a separation $s\in S$ such that both $\vr<\vs$ and $\vr<\sv$, and
\item \emph{small} if $\vr\le\rv$.
\end{enumerate}
In this paper, we will not have to worry about trivial separations.
The only degenerate separation of a set $V$ is $(V,V)$; its small separations are precisely the ones of the form $(A,V)$ with $A\subset V$.
All degenerate and trivial separations are small.

A separation system is
\begin{enumerate}
\item \emph{essential} if it contains neither degenerate nor trivial elements, and
\item \emph{regular} if it contains no small elements.
\end{enumerate}
If $(\vS,{\le},{}^\ast)$ is essential or regular, then we also call $\vS$ and $S$ \emph{essential} or \emph{regular}, respectively.
Regular implies essential.

A \emph{tree set} is a nested essential separation system.
If $(\vS,{\le},{}^\ast)$ is a tree set, then we also call $\vS$ and $S$ \emph{tree sets}.
If $T$ is a tree, then $\vE(T)$ is a tree set, the \emph{edge tree set} of~$T$.

In this paper, separations usually will not be small, and hence separation systems usually will be regular.
This means that when we define a candidate for a tree set and have to verify that it really is a tree set, it will suffice to verify nestedness unless stated otherwise.

A consistent orientation $O$ of a tree set $\vS$ is equal to the down-closure
\begin{align*}
    \dc{\sigma}_{\vS}:=\{\,\vr\in\vS\mid \exists\vs\in\sigma:\vr\le\vs\,\}
\end{align*}
in $\vS$ of the star $\sigma$ formed by the maximal elements of~$O$ if and only if every element of $O$ lies below some maximal element of~$O$.
We call these orientation defining stars $\sigma$ the \emph{splitting stars} of~$\vS$.
For example, the splitting stars of the edge tree set $\vE(T)$ of a finite tree are precisely the stars~$\vF_{\! t}$.
But if $T$ is a ray $v_0v_1\ldots$, then $\{\,(v_n,v_{n+1})\mid n\in\N\,\}$ is a consistent orientation of $\vE(T)$ that has no maximal element.

Gollin and Kneip~\cite{TreelikeSpaces} characterised the tree sets that are isomorphic to the edge tree set of a tree.
An \emph{isomorphism} between two separation systems is a bijection between their underlying sets that respects both their partial orderings and their involutions.
A chain $\mathcal{C}$ in a given poset is said to have \emph{order-type} $\alpha$ for an ordinal~$\alpha$ if $\mathcal{C}$ with the induced linear order is order-isomorphic to $\alpha$. The chain $\mathcal{C}$ is then said to be an $\alpha$-\emph{chain}.

\begin{theorem}[{\cite[Theorem~1]{TreelikeSpaces}}]\label{IVTheUndominatingStar;KneipTreeSets}
\label{KneipTreeSets}
A tree set is isomorphic to the edge tree set of a tree if and only if it is regular and contains no $(\omega+1)$-chain.
\end{theorem}

\subsection{Separations of graphs}

A \emph{separation of a graph} $G$ is a separation $\{A,B\}$ of the set $V(G)$ (meaning $A\cup B=V(G)$) such that $G$ has no edge `jumping' the \emph{separator} $A\cap B$, meaning that $G$ contains no edge between $A\setminus B$ and $B\setminus A$.
Thus, (oriented) separations of graphs are an instance of (oriented) separations of sets.
The \emph{order} of $\{A,B\}$ is the cardinal~$\vert A\cap B\vert$.
The set of all finite-order separations of a graph $G$ is denoted by \mbox{$\Sinf=\Sinf(G)$}.
A tree set \emph{of} $G$ is a tree set of separations of $G$ with the usual partial ordering and involution.

If $(A,B)$ and $(C,D)$ are two separations of $G$, then
\begin{enumerate}
\item $(A,B)\vee (C,D):=(A\cup C,B\cap D)$ is their \emph{supremum}, and
\item $(A,B)\wedge (C,D):=(A\cap C,B\cup D)$ is their \emph{infimum}.
\end{enumerate}
Supremum and infimum satisfy De~Morgan's law: $(\vr\vee\vs)^\ast=\rv\wedge\sv$.

The following non-standard notation often will be useful as an alternative perspective on separations of graphs.
Recall that for a vertex set $X\subset V(G)$ we denote the collection of the components of $G-X$ by $\cC_X$.
If any $X\subset V(G)$ and $\cC\subset\cC_X$ are given, then these give rise to a separation of $G$ which we denote by
\begin{align*} 
    \sep{X}{\cC}:=\big\{\;V\setminus V[\cC]\;,\;X\cup V[\cC]\;\big\}
\end{align*}
where $V[\cC]=\bigcup\,\{\,V(C)\mid C\in\cC\,\}$.
Note that every separation $\{A,B\}$ of $G$ with $A,B\subset V(G)$ can be written in this way. 
For the orientations of $\sep{X}{\cC}$ we write
\begin{align*}
\rsep{X}{\cC}:=\big(\;V\setminus V[\cC]\;,\;X\cup V[\cC]\;\big)\quad\text{and}\quad\lsep{\cC}{X}:=\big(\;V[\cC]\cup X\;,\;V\setminus V[\cC]\;\big).
\end{align*}
If $C$ is a component of $G-X$ we write $\sep{X}{C}$ instead of $\sep{X}{\{C\}}$.  Similarly, we write $\lsep{C}{X}$ and $\rsep{X}{C}$ instead of $\lsep{\{C\}}{X}$ and $\rsep{X}{\{C\}}$, respectively.

\subsection{Parts and torsos}\label{subsec:TreeSets}





If $T$ is a tree set of separations of $G$ and $O$ is a consistent orientation of $T$, then the intersection $\Pi=\bigcap\,\{\,B\mid (A,B)\in O\}$ is called the \emph{part} of $O$.
And the graph that is obtained from $G[\Pi]$ by adding an edge $xy$ whenever $x\neq y\in\Pi$ lie together in the separator of some separation of $O$ is called the \emph{torso} of $O$ (or of $\Pi$ if $O$ is clear from context).
We denote the torso of $O$ by $\torso(G,O)$.

We will need the following lemma and its corollaries (the lemma is folklore and has been proved, e.g., in~\cite{carmesin2014all}; we present an alternative proof for convenience):

\begin{lemma}\label{partPathsSeparator}
If $\Pi$ is a part of a tree set of $G$, then for every $G[\Pi]$-path $P$ there is some separation of the tree set whose separator contains both endvertices of $P$.
\end{lemma}
\begin{proof}
Let $O$ be any consistent orientation of a tree set of $G$, write $\Pi$ for its part and suppose that $P=x v_1\ldots v_n y$ is a $G[\Pi]$-path (so $n\ge 1$).
For every $k\in [n]$ pick an oriented separation $(A_k,B_k)\in O$ with $v_k\in A_k\setminus B_k$ (so that $(A_k,B_k)$ witnesses $v_k\notin\Pi$).
Let $N$ consist of the ${\le}$-maximal separations from the collection $\{\,(A_k,B_k)\mid k\in [n]\}$.
Then for every $v_k$ there is a separation $(A,B)\in N$ with $v_k\in A\setminus B$.
Our aim is to show that $N$ is a singleton, since then the separator of the sole separation in $N$ must contain both $x$ and $y$, so we would be done.
By the choice of $N$, every two oriented separations in $N$ are ${\le}$-incomparable.
As $O$ is a consistent orientation of a tree set, this means that $N$ must be a star.
Then $\vert N\vert=1$ is evident, since otherwise the sides $G[A\setminus B]$ for $(A,B)\in N$ altogether induce a disconnection of the subpath $v_1\ldots v_n$ of $P$ contradicting its connectedness.
\end{proof}

\begin{corollary}\label{endInClosureOfPartGivesRayInTorso}
    If $\Pi$ is a part of a tree set of $G$ and $\omega$ is an end of $G$ in the closure of $\Pi$ while $G[\Pi]$ coincides with the torso of $\Pi$, 
    then $\omega$ has a ray in $G[\Pi]$.
\end{corollary}
\begin{proof}
    If $\omega$ lies in the closure of $\Pi$, we find a comb in $G$ with its spine $R$ in $\omega$ and all of its teeth in $\Pi$.
    Without loss of generality the comb meets $\Pi$ precisely in its teeth.
    Then, as $G[\Pi]$ coincides with the torso of $\Pi$, it has an edge between every two consecutive teeth by Lemma~\ref{partPathsSeparator}, and so contains a ray equivalent to $R$.
\end{proof}

\begin{corollary}\label{equivalentRaysOfPartToTorso}
If $\Pi$ is a part of a tree set of $G$ and two rays of $G[\Pi]$ are equivalent in $G$, then they are equivalent in the torso of $\Pi$ as well.
\end{corollary}
\begin{proof}
Given two rays of $G[\Pi]$ that are equivalent in $G$, we inductively construct infinitely many pairwise vertex-disjoint paths in $G$ between them, and then employ Lemma~\ref{partPathsSeparator} to turn these into paths of the torso.
\end{proof}

The next corollary has already been known to Carmesin~\cite{carmesin2014all}:

\begin{corollary}\label{restrictconnectedsets}
The intersection of a connected set of vertices of $G$ with a part of a tree set of separations of $G$ induces a connected subgraph of the part's torso.\qed
\end{corollary}


\subsection{Infinite tangles}
\label{DefInfOrderTangles}
\label{PropertiesOfInfOrderTangles}


The \emph{interior} of a star $\{\,(A_i,B_i)\mid i\in I\,\}\subset\vSinf$ is the intersection $\bigcap_{i\in I}B_i$.
An \emph{$\aleph_0$-tangle} (of $G$), or \emph{infinite tangle}, is a consistent orientation of $\Sinf$ that contains no finite star of finite interior as a subset.
We write $\Theta=\Theta(G)$ for the set of all $\aleph_0$-tangles of~$G$.
This particular definition is due to Diestel who showed in his paper~\cite{EndsAndTangles} that it is equivalent to the original definition by Robertson and Seymour~\cite{GMX}.
Infinite tangles are resilient in the following sense:
\begin{lemma}[{\cite[Lemma~1.10]{EndsAndTangles}}]\label{FiniteEditDistance}
Let $\tau$ be an $\aleph_0$-tangle of $G$ and $(A,B)\in\tau$.
Let $(A',B')$ be a separation of $G$ with $A\triangle A'$ and $B\triangle B'$ finite.
Then $(A',B')\in\tau$.
\end{lemma}

In the remainder of this section, we give a summary of the results in Diestel's paper~\cite{EndsAndTangles} on infinite tangles.
If $\omega$ is an end of $G$, then letting
\begin{align*}
\tau_\omega:=\big\{\,\rsep{X}{\cC}\in\vSinf\;\big\vert\;C(X,\omega)\in\cC\,\big\}
\end{align*}
defines an injection $\Omega\hookrightarrow\Theta$, $\omega\mapsto\tau_\omega$. The $\aleph_0$-tangles of the form $\tau_\omega$ are called \emph{end tangles}.
By abuse of notation we write $\Omega$ for the collection of all end tangles of $G$, so we have $\Omega\subseteq\Theta$.

In order to understand the $\aleph_0$-tangles that are not ends, Diestel studied an inverse limit description of~$\Theta$.
In this paper, partition classes are required to be non-empty as usual, with the exception that whenever we speak of a \emph{bipartition} we allow for at most one empty class.
Now if $\tau$ is an $\aleph_0$-tangle, then for every $X\in\cX$ it chooses one \emph{big side} from each bipartition $\{\cC,\cC'\}$ of $\cC_X$, namely the $\cD\in\{\cC,\cC'\}$ with $\rsep{X}{\cD}\in\tau$.
Since it chooses theses sides consistently, it induces an ultrafilter $U(\tau,X)$ on $\cC_X$, one for every $X\in\cX$, which is given by
\begin{align*}
U(\tau,X)=\{\,\cC\subseteq\cC_X\mid \rsep{X}{\cC}\in\tau\,\},
\end{align*}
and these ultrafilters are compatible in that they form a limit of the inverse system $\{\,\beta(\cC_X)\,,\,\beta(\cb_{X',X})\,,\,\cX\,\}$.
Here, each set $\cC_X$ is endowed with the discrete topology and $\beta(\cC_X)$ denotes its \SC\ \comp .
Every bonding map $\beta(\cb_{X',X})$ is the unique continuous extension of $\cb_{X',X}$ that is provided by the \SC\ property.
More explicitly, the map $\beta(\cb_{X',X})$ sends each ultrafilter $U'\in\beta (\cC_{X'})$ to its \emph{restriction}
\begin{align*}
U'\rest X=\{\,\cC\subseteq\cC_X\mid\exists\,\cC'\in U'\colon\cC\supseteq\cC'\rest X\,\}\in\beta (\cC_X)
\end{align*}
where $\cC'\rest X=\cb_{X',X}[\cC']$.
As one of his main results, Diestel showed that the map
\begin{align*}
\tau\mapsto (\,U(\tau,X)\mid X\in\cX\,)
\end{align*}
defines a bijection between the tangle space $\Theta$ and the inverse limit $\invLim\beta(\cC_X)$.
Moreover, he showed that the ends of $G$ are precisely those $\aleph_0$-tangles whose induced ultrafilters are all principal.

For every $\aleph_0$-tangle $\tau$ we write $\cX_\tau$ for the collection of all $X\in\cX$ for which the induced ultrafilter $U(\tau,X)$ is free.
Equivalently, $\cX_\tau$ is the collection of those $X\in\cX$ for which the star $\{\,\lsep{C}{X}\mid C\in \cC_X\,\}$ is included in $\tau$.
The set $\cX_\tau$ is empty if and only if $\tau$ is an end tangle. 
An $\aleph_0$-tangle $\tau$ with $\cX_\tau$ non-empty is called an \emph{ultrafilter tangle}, and we write $\Upsilon$ for the collection of all ultrafilter tangles, i.e.~$\Upsilon=\Theta\setminus\Omega$.
\begin{theorem}[{\cite[Theorem~3.5]{EndsAndTangles}}]\label{UFtangleDeterminedByFreeUF}
For every ultrafilter tangle $\tau$ and each $X\in\cX_\tau$ the free ultrafilter $U(\tau,X)$ determines $\tau$ in that
\begin{align*}
    \tau=\big\{\,(A,B)\in\vSinf{}\;\big\vert\;\exists\,\cC\in U(\tau,X):V[\cC]\subset B\setminus A\,\big\}.
\end{align*}
\end{theorem}

For every ultrafilter tangle $\tau$ the set $\cX_\tau\subset\cX$ has a least element $X_\tau$ of which it is the up-closure $\cX_\tau=\uc{X_\tau}_{\cX}:=\{\,X\in\cX\mid X\supset X_\tau\,\}$.
These elements have been characterised combinatorially in~\cite{EndsTanglesCrit} as follows.
Given $X\subset V(G)$ and a subset $Y\subset X$ we write $\cC_X(Y)$ for the set $\{\,C\in\cC_X\mid N(C)=Y\,\}$ of components of $G-X$ that have their neighbourhood precisely equal to $Y$.
In the special case of $X=Y$ we abbreviate $\cC_X(X)$ to $\CX$.
A finite vertex set $X\in\cX$ is \emph{critical} if $\CX$ is infinite.
The collection of all the critical vertex sets of a graph $G$ is denoted by $\crit(G)$.
A vertex set $X\subset V(G)$ is of the form $X_\tau$ for an ultrafilter tangle $\tau$ of~$G$ if and only if $X$ is critical.
But more is true:
The critical vertex sets allow us to describe the ultrafilter tangles explicitly, like the ends allow us to describe the end tangles explicitly.
An \emph{ultrafilter tangle blueprint} is an ordered pair $(X,U)$ of a critical vertex set $X$ and a free ultrafilter $U$ on~$\CX$.

\begin{theorem}[{\cite[Theorem~4.10]{EndsTanglesCrit}}]
Let $G$ be any graph. Then the map
\begin{align*}
    \tau\mapsto (\,X_{\tau}\,,\,U(\tau,X_\tau)\cap 2^{\CC{X_\tau}})
\end{align*}
is a bijection between the ultrafilter tangles and the ultrafilter tangle blueprints.
In particular, $\rsep{X_\tau}{\CC{X_\tau}}\in\tau$ for every ultrafilter tangle $\tau$ of~$G$.
\end{theorem}

\noindent We will resume the following notation from~\cite{EndsTanglesCrit} for critical vertex sets.
For every $X\in\cX$ and all critical $Y$ that are not entirely contained in $X$ we write $C_X(Y)$ for the unique component of $G-X$ meeting $Y$ (equivalently: including $\bigcup\cC_{X\cup Y}(Y)$).

\begin{lemma}[{\cite[Lemma~4.8]{EndsTanglesCrit}}]\label{UFtanglePrincipalGenerator}
For every ultrafilter tangle $\tau$ and each $X\in\cX\setminus\cX_\tau$ we do have $X_\tau\subset X\cup C_X(X_\tau)$ and the ultrafilter $U(\tau,X)$ is generated by $\{C_X(X_\tau)\}$.
\end{lemma}

Next, we describe Diestel's tangle \comp .
For this, we recall the notion of the 1-complex of a graph~$G$.
In the \emph{1-complex} of $G$ which we denote also by $G$, every edge $e=xy$ is a homeomorphic copy $[x,y]:=\{x\}\sqcup\mathring{e}\sqcup\{y\}$ of $[0,1]$ with $\mathring{e}$ corresponding to $(0,1)$. 
The point set of $G$ is $V\sqcup\bigsqcup_{e\in E}\mathring{e}$.
Points in $\mathring{e}$ are called \emph{inner edge points}, and they inherit their basic open neighbourhoods from $(0,1)$. 
For each subcollection $F\subseteq E$ we write $\mathring{F}$ for the set $\bigsqcup_{e\in F}\mathring{e}$ of inner edge points of edges in $F$.
The basic open neighbourhoods of a vertex $v$ of $G$ are given by unions $\bigcup_{e\in E(v)}[v,i_e)$ of half open intervals with each $i_e$ some inner edge point of $e$ where $E(v)$ denotes the set of edges of $G$ at $v$.

To obtain the tangle \comp\ $\modG$ of a graph $G$ we extend the 1-complex of $G$ to a topological space $G\sqcup\Theta=G\sqcup\invLim{}\beta (\cC_X)$ by declaring as open in addition to the open sets of $G$, for all $X\in\cX$ and all $\cC\subseteq\cC_X$, the sets
\begin{align*}
\cO_{\modG}(X,\cC):=\medcup\cC\cup\mathring{E}(X,\medcup\cC)\cup\big\{\,(\,U_Y\colon Y\in\cX\,)\in\invLim{}\beta(\cC_X)\;\big\vert\; \cC\in U_X\,\big\}
\end{align*}
and taking the topology this generates.
\begin{theorem}[{\cite[Theorem 1]{EndsAndTangles}}]
\label{DiestelsTangleCompWorks}
Let $G$ be any graph.
\begin{enumerate}
\item $\modG$ is a \comp\ of $G$ with totally disconnected remainder.
\item If $G$ is locally finite and connected, then $\modG$ coincides with the Freudenthal \comp\ of $G$.
\end{enumerate}
\end{theorem}

\begin{theorem}[{\cite{StoneCechTangles}}]
The tangle \comp\ of any graph $G$ is obtained from its \SC\ \comp\ $\beta G$ by first declaring $G$ to be open\footnote{When $G$ is locally compact, it is automatically open in $\beta G$, and so this step is redundant for locally finite graphs.} in $\beta G$ and then collapsing each connected component of the \SC\ remainder to a single point.
\end{theorem}

\subsection{Trees-of-tangles}

A separation $\{A,B\}$ of $G$ and its orientations \emph{distinguish} two infinite tangles $\tau_1$ and $\tau_2$ of $G$ if $\tau_1$ and $\tau_2$ orient $\{A,B\}$ differently, i.e., if $(A,B)\in\tau_1$ and $(B,A)\in\tau_2$ or vice versa.
The separation $\{A,B\}$ distinguishes $\tau_1$ and $\tau_2$ \emph{efficiently} if it has minimal order $\vert A\cap B\vert$ among all the separations of $G$ that distinguish $\tau_1$ and $\tau_2$.
The result by Carmesin that implies the tree-of-tangles theorem for the infinite tangles of locally finite infinite graphs states:

\begin{theorem}[{\cite[Corollary~5.17]{carmesin2014all}}]\label{CarmesinEndsTreeSet}
Every connected graph $G$ has a tree set of finite-order separations of $G$ that efficiently distinguishes all the ends of $G$.
\end{theorem}

\noindent Here, we view the ends as end tangles.

Very recently, Carmesin, Hamann and Miraftab showed a canonical version of this theorem; see their paper~\cite{CanonicalTreeOfTangles}.
We should also mention the work by Dunwoody and Krön~\cite{VertexCuts} and the work by Elbracht, Kneip and Teegen~\cite{InfiniteSplinters}.

\subsection{Combinatorial indistinguishability and tame separations}

We say that two infinite tangles $\tau_1$ and $\tau_2$ are \emph{combinatorially distinguishable} if at least one of them is an end tangle or they are both ultrafilter tangles but such that their critical vertex sets $X_{\tau_1}$ and $X_{\tau_2}$ are distinct.
Thus, when $\tau_1$ and $\tau_2$ are combinatorially indistinguishable, they are ultrafilter tangles with $X_{\tau_1}=X_{\tau_2}$.
Then we also call them \emph{equivalent} for short and write $\tau_1\sim\tau_2$.
There exist separations of $G$ that do not distinguish any two equivalent tangles:


\begin{definition}
A finite-order separation $\sep{X}{\cC}$ of $G$ and its orientations are \emph{tame} if for no $Y\subset X$ both $\cC_X(Y)\cap\cC$ and $\cC_X(Y)\cap (\cC_X\setminus\cC)$ are infinite.
We write $\St=\St(G)$ for the set of all tame finite-order separations of $G$.
\end{definition}

\begin{lemma}[{\cite[§5]{EndsTanglesCrit}}]
If $\tau_1$ and $\tau_2$ are two equivalent infinite tangles of a graph~$G$, then $\tau_1\cap\vSt=\tau_2\cap\vSt$.
\end{lemma}

\noindent The tree set that we construct in the proof of Theorem~\ref{TreeSetForInfTangles} will consist of tame separations.

Identifying all equivalent $\aleph_0$-tangles yields the quotient $\Theta/{\sim}=\Omega\sqcup\crit(G)$ which is yet again a tangle space.
An $\aleph_0$-tangle \emph{of $\St$} is a consistent orientation of $\St$ that contains no finite star of finite interior as a subset.
We write $\Tt=\Tt(G)$ for the set of all $\aleph_0$-tangles of~$\St(G)$.

\begin{theorem}[{\cite[Theorem~5.10]{EndsTanglesCrit}}]
Let $G$ be any graph.
The $\aleph_0$-tangles of $\St$ are precisely the ends and critical vertex sets of $G$, i.e.~$\Tt(G)=\Omega(G)\sqcup\crit(G)$.
\end{theorem}

\noindent As a consequence, every graph is compactified by its ends and critical vertex sets in the tangle-type \comp\ $\modG/{\sim}=G\sqcup\Omega(G)\sqcup\crit(G)=G\sqcup\Tt(G)$.

Since the tree set that we construct in the proof of Theorem~\ref{TreeSetForInfTangles} will efficiently distinguish all the combinatorially distinguishable $\aleph_0$-tangles of $G$ and consist of tame separations, it will efficiently distinguish all the \mbox{$\aleph_0$-tangles} of~$\St$ (i.e.\ all the ends and critical vertex sets of~$G$).



\section{Example section}\label{sec:examples}

\noindent The aim of this section is twofold.
First, we verify that our main result, Theorem~\ref{TreeSetForInfTangles}, is indeed best possible as claimed in the introduction.
More precisely, in Subsection~\ref{UltrafiltersAndTreeSets} we show that tree sets of finite-order separations cannot distinguish all the ultrafilter tangles from the same equivalence class at once---for any~$G$.

Second, we study the candidate for a starting tree set that is formed by the separations $\csep{X}$ with $X$ critical in $G$ (recall that these are precisely the separations which naturally accompany the ultrafilter tangles).
More precisely, in Subsection~\ref{TheProblemCase} we will see two example graphs showing that it is necessary to modify the tree set candidate:
For the first example graph, the separations $\csep{X}$ form a tree set but do not distinguish any two ultrafilter tangles at all.
For the second example graph, the separations $\csep{X}$ are not even nested.

\subsection{Ultrafilters and tree sets}\label{UltrafiltersAndTreeSets}

In this subsection we show that, as soon as a graph $G$ has some ultrafilter tangle $\tau$, it already cannot admit a tree set of finite-order separations that distinguishes all the ultrafilter tangles that are equivalent to~$\tau$.
As our first step, we translate the problem from graphs to bipartitions of sets.

For this, we need to make some things formal first.
Suppose that $K$ is a non-empty set. We let $\vcB(K):=2^K$. Thus, every subset of $K$ is an oriented `separation'.
The partial ordering ${\le}$ of $\vcB(K)$ will be ${\supset}$, the involution ${}^\ast$ on $\vcB(K)$ will be complementation in the set $K$.
If desired, we can think of a separation $Z\subset K$ as the oriented bipartition $(Z^\ast,Z)$ of $K$, and then $\cB(K)$ is the set of bipartitions of~$K$.
Note that two separations $Z_1,Z_2\in\vcB(K)$ are nested if $Z_1\subset Z_2$ or $Z_1\supset Z_2$ or $Z_1\cup Z_2=K$ or $Z_1\cap Z_2=\emptyset$.
A \emph{tree set of bipartitions} of $K$ is a tree set contained in $\vcB(K)$ with the induced partial ordering and involution.
Note that $\emptyset$ is the sole small separation in $\vcB(K)$ for $Z\subset K\setminus Z$ implies $Z=\emptyset$.
Since an ultrafilter on $K$ happens to be an orientation of $\cB(K)$, a tree set $\vT$ of bipartitions of $K$ distinguishes two distinct ultrafilters $U\neq U'$ on $K$ if there is some $Z\in \vT$ with $Z\in U$ and $Z^\ast\in U'$.
We are almost ready for the translation, we only need one more lemma:

\begin{lemma}\label{ufTangleCofinalSeps}
Let $\tau$ be any ultrafilter tangle of $G$ with blueprint $(X,U)$ and let any separation $\rsep{Y}{\cD}\in\tau$ be given.
Write $\cC$ for the set of those components in $\CX$ that avoid~$Y$.
\begin{enumerate}
    \item If $Y$ includes $X$, then $\rsep{Y}{\cD}\le\rsep{X}{\cD\cap\CX}\in\tau$.
    \item Otherwise $\rsep{Y}{\cD}\le\rsep{X}{\cC}\in\tau$.
\end{enumerate}
In particular, the set $\big\{\,\rsep{X}{\cC}\;\big\vert\;\cC\in U\,\big\}$ is cofinal in $\tau$.
\end{lemma}
\begin{proof}
Since $Y$ is finite, $\cC$ is a cofinite subset of $\CX$, giving $\cC\in U$.

(i)
The intersection $\cD\cap\CX$ can be written in a more complicated way as \mbox{$(\cD\rest X)\cap\cC$} where $\cC\in U$ as noted above and
\begin{align*}
    \cD\rest X=\{\,C\in\cC_X\mid\exists\,D\in\cD:D\supset C\,\}\in U.
\end{align*}
Hence $\rsep{X}{\cD\cap\CX}\in\tau$.
It is straightforward to check $\rsep{Y}{\cD}\le\rsep{X}{\cD\cap\CX}$.

(ii)
From $\cC\in U$ we get $\rsep{X}{\cC}\in \tau$.
Lemma~\ref{UFtanglePrincipalGenerator} deduces from $\rsep{Y}{\cD}\in\tau$ that $C_Y(X)\in\cD$. 
Finally, we calculate $\rsep{Y}{\cD}\le\rsep{Y}{C_Y(X)}\le\rsep{X}{\cC}$ where for the second inequality we use that every component in $\cC$ sends an edge to the non-empty $X\setminus Y\subset C_Y(X)$ to deduce $\bigcup\cC\subset C_Y(X)$.
\end{proof}

Now we are ready for the translation:

\begin{lemma}\label{treeSetUFtanglesToTreeSetUFonN}
Let $X$ be a critical vertex set of $G$. Then every tree set of finite-order separations of $G$ that distinguishes all the ultrafilter tangles $\tau$ of $G$ with $X_\tau=X$ does induce a tree set of bipartitions of $\CX$ that distinguishes all the free ultrafilters on $\CX$.
\end{lemma}

\begin{proof}
Let $\vT$ be a tree set of finite order separations of $G$ that distinguishes all the ultrafilter tangles of $G$ with $X_\tau=X$.
Without loss of generality every separation $\rsep{Y}{\cD}\in\vT$ distinguishes some two such ultrafilter tangles, and so $X\subset Y$ follows for all $\rsep{Y}{\cD}\in\vT$.

The candidate for a tree set of bipartitions of $\CX$ is $\{\,\bar{\cD}\mid \rsep{Y}{\cD}\in\vT\,\}$ where $\bar{\cD}=\cD\cap\CX$.
But when $\rsep{Y}{\cD'}$ is the inverse of $\rsep{Y}{\cD}$ it can happen that $\bar{\cD}'$ is not the inverse of $\bar{\cD}$ in $\vcB(\CX)$.
For example, this happens when a finite component $C\in\CX$ is contained in $Y$, for then both $\bar{\cD}'$ and $\bar{\cD}$ are missing $C$.

We overcome this obstacle as follows.
First, we choose any consistent orientation $O$ of $\vT$ (such an orientation exists, e.g., by~\cite[Lemma~4.1]{AbstractSepSys} which essentially applies Zorn's lemma to achieve this).
Then, we define $N_O:=\{\,\bar{\cD}\mid \lsep{\cD}{Y}\in O\,\}$.
Finally, we claim that $\vN:=N_O\cup N_O^\ast$ is a tree set of bipartitions of $\CX$ that distinguishes all the free ultrafilters on $\CX$.

To verify that $\vN$ is a tree set we show that $\vN$ is nested.
For this, consider any two separations $\lsep{\cD_1}{Y_1},\lsep{\cD_2}{Y_2}\in O$.
Then, say, either $\lsep{\cD_1}{Y_1}\le\lsep{\cD_2}{Y_2}$ implies $\bar{\cD}_1\subset\bar{\cD}_2$ or $\lsep{\cD_1}{Y_1}\le\rsep{Y_2}{\cD_2}$ implies $\cD_1\subset(\bar{\cD}_2)^\ast$.
So $\vN$ is a tree set.

Now let $U\neq U'$ be any distinct two free ultrafilters on $\CX$.
Then there is a separation $\rsep{\cD}{Y}\in O$ that distinguishes the ultrafilter tangles $\tau_U$ and $\tau_{U'}$ corresponding to $(X,U)$ and $(X,U')$, say with $\lsep{\cD}{Y}\in\tau_U$ and $\rsep{Y}{\cD}\in\tau_{U'}$.
By Lemma~\ref{ufTangleCofinalSeps} we have $\bar{\cD}\in U'$.
Similarly $\overline{\cD_Y\setminus\cD}\in U$, which then via the inclusion $\overline{\cD_Y\setminus\cD}\subset (\bar{\cD})^\ast$ implies $(\bar{\cD})^\ast\in U$.
\end{proof}

As a consequence of this lemma, it suffices to show

\begin{theorem}\label{BipsUFsNoTreeSetDists}
If $K$ is an infinite set, then no tree set of bipartitions of $K$ distinguishes all the free ultrafilters on $K$.
\end{theorem}

in order to obtain our desired result:

\begin{corollary}
If $\tau$ is an ultrafilter tangle of $G$, then no tree set of finite-order separations of $G$ distinguishes all the ultrafilter tangles that are equivalent to $\tau$.\qed
\end{corollary}

Theorem~\ref{BipsUFsNoTreeSetDists} above has been proved independently from us by Bowler~\cite{NathanDistinguishUltrafilters} in 2014 who did not publish his findings.
The proof presented below is ours.
For the proof we need the following lemma which is a tree set version of the fact that every connected infinite graph contains either a ray or a vertex of infinite degree, \cite[Proposition~8.2.1]{BibleNew}:

\begin{lemma}\label{StarCombForTreeSets}
Every regular infinite tree set contains either an $\omega$-chain or an infinite splitting star.
\end{lemma}
\begin{proof}
If a tree set contains no $\omega$-chain, then it is isomorphic to the edge tree set of a rayless tree by Theorem~\ref{KneipTreeSets}.
This tree, then, must have an infinite degree vertex if the tree set is infinite.
\end{proof}

If $U$ is an ultrafilter on a set $K$ and $\cK$ is a partition of $K$, then we write $U\bdot \cK$ for the \emph{induced} ultrafilter on $\cK$ given by $\{\,\cA\subset\cK\mid\bigcup\cA\in U\,\}$.
Notably, if $U$ is principal, then so is $U\bdot\cK$.
Conversely, every ultrafilter $\cU$ on $\cK$ gives a filter
\begin{align*}
    \lfloor\,\{\,\medcup\cA\mid\cA\in\cU\,\}\,\rfloor_K:=\{\,A\subset K\mid\exists\,\cA\in\cU:A\supset\medcup\cA\,\}
\end{align*}
on $K$, and every ultrafilter $U$ on $K$ that extends this filter induces $\cU$ in that $\cU=U\bdot \cK$.
Phrased differently, the map $U\mapsto U\bdot\cK$ is a surjection from the set of ultrafilters on $K$ onto the set of ultrafilters on $\cK$.
Notably, free ultrafilters on $\cK$ are induced only by free ultrafilters on $K$.

\begin{proof}[Proof of Theorem~\ref{BipsUFsNoTreeSetDists}]
Let any infinite set $K$ be given and assume for a contradiction that $\vT$ is a tree set of bipartitions of $K$ that distinguishes all the free ultrafilters on $K$.
If $\vT$ is finite, then there are only finitely many orientations of $\vT$.
But there are infinitely many free ultrafilters on $K$, so a finite tree set cannot possibly distinguish all of them.
Therefore, $\vT$ must be infinite.
Since the empty set does not distinguish any two ultrafilters on $K$ we may assume without loss of generality that $\vT$ is regular.
Then by Lemma~\ref{StarCombForTreeSets} we know that $\vT$ contains either an $\omega$-chain or an infinite splitting star.

Suppose first that $\vT$ contains an $\omega$-chain; that is to say that we find a sequence $(Z_n)_{n<\omega}$ in $\vT$ with $Z_n\supsetneq Z_{n+1}$ for all $n$.
As $\vT$ is a tree-set, $K\setminus Z_0$ is non-empty.
Put $Z_\omega:=\bigcap_{n<\omega}Z_n$.
Then $Z_\omega$ is nested with every separation in $\vT$.
More precisely, every separation in $T$ has an orientation $Z$ such that either $Z\supseteq Z_n$ for some $n<\omega$ or $Z_\omega\supseteq Z$.
We turn the transfinite sequence $(Z_\alpha)_{\alpha\le\omega}$ into a partition of $K$, as follows.
For every $n<\omega$ set $K_n=Z_n\setminus Z_{n+1}$; and put $K_\omega:=(K\setminus Z_0)\cup Z_\omega$.
Then $\cK:=\{\,K_\alpha\mid\alpha\le\omega\}$ is an infinite partition of $K$.
Let $\cU$ be any free ultrafilter on~$\cK$, and pick some free ultrafilter $U$ on $K$ with $\cU=U\bdot\cK$.
The free ultrafilter $\cU$ contains all cofinite subsets $\{\,K_m\mid n\le m<\omega\}\subset\cK$ with $n<\omega$, and so $U$ contains all $Z_n\setminus Z_\omega$ with $n<\omega$.
Recall that every separation in $T$ has an orientation $Z$ such that either $Z\supseteq Z_n$ for some $n<\omega$ or $Z_\omega\supseteq Z$.
Hence for every separation $\{Z^\ast,Z\}\in T$ we have that either $Z\supseteq Z_n$ with $Z_n\setminus Z_\omega\in U$ implies $Z\in U$, or $Z_\omega\supseteq Z$ with $Z_0\setminus Z_\omega\in U$ implies $Z^\ast\in U$.
Therefore, if $\cU'$ is any free ultrafilter on $\cK$ other than $\cU$, and $U'$ is a free ultrafilter on $K$ inducing $\cU'$, then $U'$ orients every separation in $T$ the same way as $U$.
But then $\vT$ does not distinguish $U$ and $U'$ from each other, a contradiction.

Finally suppose that $T$ contains an infinite splitting star $\sigma=\{\,K_i\mid i\in I\}$.
If $\cK:=\{\,K_i^\ast\mid i\in I\}$ is not yet a partition of $K$, then we add the non-empty interior $\bigcap_{i\in I}K_i$ of $\sigma$ to $\cK$ to turn $\cK$ into one.
Let $\cU$ be any free ultrafilter on $\cK$, and pick some free ultrafilter $U$ on $K$ inducing $\cU$.
The free ultrafilter $\cU$ contains all collections $\cK-K_i^\ast$, and hence $U$ contains all $K_i$.
Now every separation in $T$ has an orientation $Z$ with $Z\supseteq K_i$ for some $i\in I$ as $\sigma$ is splitting, and then $K_i\in U$ implies $Z\in U$.
Therefore, if $\cU'$ is any free ultrafilter on $\cK$ other than $\cU$, and $U'$ is a free ultrafilter on $K$ inducing $\cU'$, then $U'$ orients every separation in $T$ the same way as $\cU$.
But then $\vT$ does not distinguish $U$ and $U'$, a contradiction.
\end{proof}

We remark that the proof above even shows the following stronger version of Theorem~\ref{BipsUFsNoTreeSetDists}: \emph{If $K$ is an infinite set, then for every tree set of bipartitions of $K$ there is a collection of at least $2^{2^{\aleph_0}}=2^\frakc$ many free ultrafilters on $K$ all of which induce the same orientation of the tree set.}\footnote{By improving Lemma~\ref{StarCombForTreeSets} it might be possible to replace $2^{2^{\aleph_0}}$ with $2^{2^{\vert K\vert}}$.}
So if $G$ has precisely one critical vertex set $X$ with $\CX$ countable, then for every tree set of finite order separations of $G$ there is a collection $\cO$ of ultrafilter tangles of $G$ such that all ultrafilter tangles in $\cO$ induce the same orientation of the tree set and the cardinal $\vert\cO\vert$ is equal to the total number $2^\frakc$ of ultrafilter tangles of $G$.

\subsection{The problem case}\label{TheProblemCase}

This subsection is dedicated to examples of graphs whose critical vertex sets give a very bad starting set
\begin{align*}
\big\{\,\csep{X}\;\big\vert\;X\in\crit(G)\,\big\}.
\end{align*}
In both cases, all the critical vertex sets interact with each other in a particular way, made precise as follows.
Let us say that two critical vertex sets $X$ and $Y$ of~$G$ \emph{form a problem case} if $X$ and $Y$ are incomparable as sets and additionally both $C_X(Y)\in\CX$ and $C_Y(X)\in\CY$ hold.


\begin{example}
Let $G$ be the subtree of the infinitely branching tree $T_{\aleph_0}$ that consists of the first three levels (for an arbitrarily chosen root).
Note that the critical vertex sets of any tree are precisely the singletons formed by its infinite degree nodes.
Then the collection
\begin{align*}
\big\{\,\crsep{X}\;\big\vert\; X\in\crit(G)\,\big\}
\end{align*}
is an infinite star of small separations.
As every $\aleph_0$-tangle contains all the small separations $(A,V)$ with $A$ finite (because these can be written as $\rsep{A}{\cC_A}$ and $\cC_A\in U(\tau,A)$ for every $\aleph_0$-tangle $\tau$), it follows that every ultrafilter tangle contains this star as a subset, and so no two ultrafilter tangles are distinguished by this star's underlying tree set.
Notably, every two distinct critical vertex sets of $G$ form a problem case.\qed
%
%
\end{example}

\begin{figure}[ht]
	\centering
	\includegraphics[scale=1]{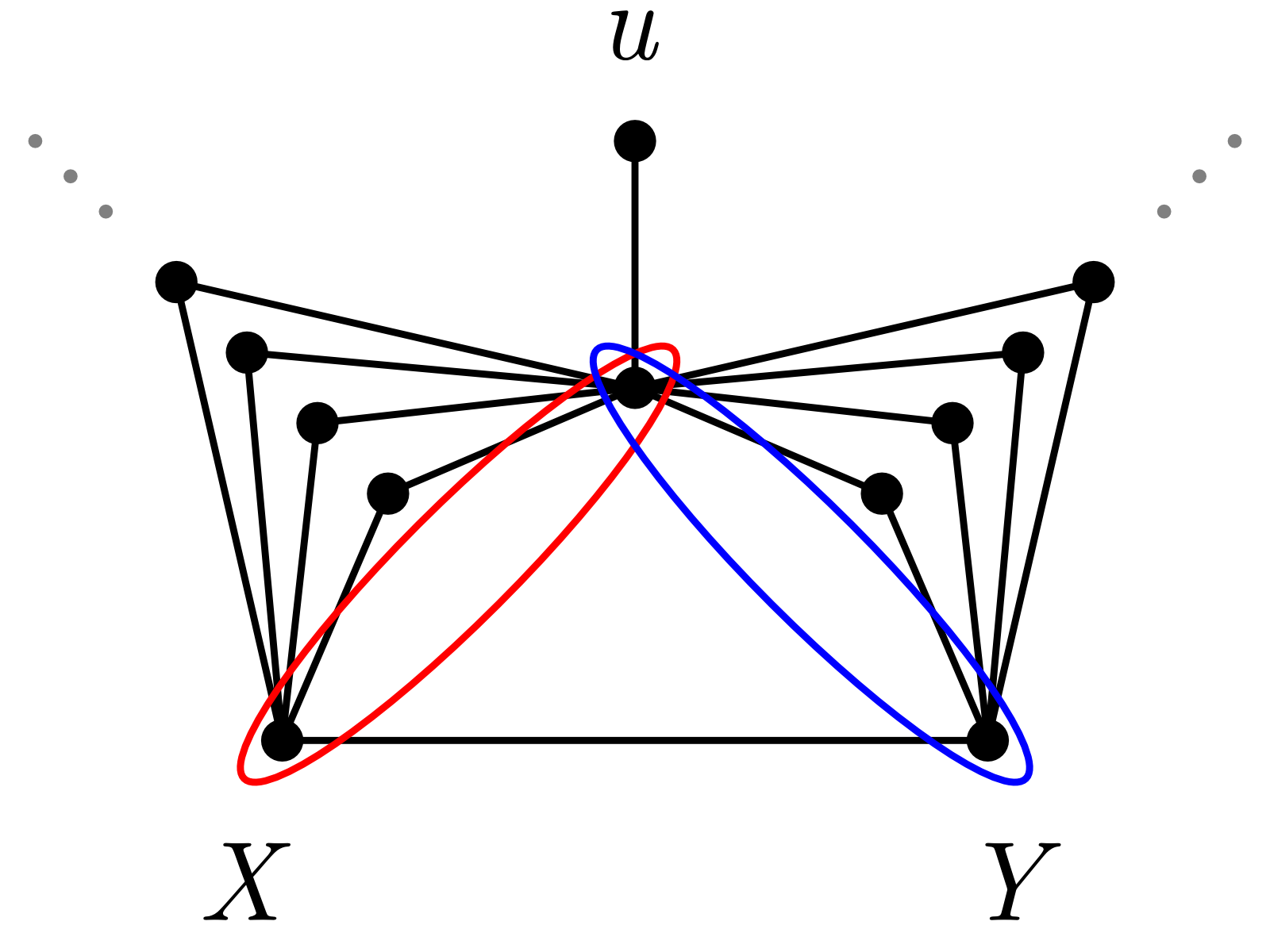}
	\caption{This graph's critical vertex sets do not give nested separations.}
	\label{fig:incomparable}
\end{figure}

\begin{example}
If $G$ is the graph shown in Figure~\ref{fig:incomparable}, then the collection
\begin{align*}
\big\{\,\csep{X}\;\big\vert\;X\in\crit(G)\,\big\}
\end{align*}
is not even nested.
Indeed, $X$ and $Y$ are the only two critical vertex sets of $G$.
Write $V'$ for $V-u$.
Then $\csep{X}=\{X+u,V'\}$ and $\csep{Y}=\{Y+u,V'\}$.
Now these two separations cannot be nested:
as $X$ and $Y$ are incomparable as sets, we have neither $(X+u,V')\le (Y+u,V')$ nor $(V',X+u)\le (V',Y+u)$.
But $(X+u,V')\le (V',Y+u)$ and $(V',X+u)\le (Y+u,V')$ are impossible as well since $X+u$ and $Y+u$ are both incomparable with $V'$ as sets.
As in the previous example we note that $X$ and $Y$ form a problem case.\qed
\end{example}

\section{The overall proof strategy}\label{sec:strategy}

\noindent Our overall strategy to achieve our main result, Theorem~\ref{TreeSetForInfTangles}, roughly goes as follows.
Let $G$ be any infinite connected graph.
Recall that every ultrafilter tangle $\tau=(X,U)$ of $G$ naturally comes with a finite-order separation $\rsep{X}{\CX}\in\tau$.
As our first step, we carefully extend and refine the set of these separations into a starting tree set $T$ that already distinguishes all the inequivalent ultrafilter tangles of $G$, but does not necessarily do so efficiently.

Next, we modify the torsos of $T$ so that every $\aleph_0$-tangle of $G$ is represented in every modified torso by some end of that modified torso.
We then show the following assertion (also see Figure~\ref{fig:simpleReflection}):
\emph{Let $\tau_1$ and $\tau_2$ be any two inequivalent $\aleph_0$-tangles of $G$ which are not efficiently distinguished by the starting tree set $T$.
For every separator $Z$ efficiently separating the $\tau_i$ in $G$ there is a modified torso $H$ of $T$ in which the ends $\eta_i$ representing the tangles $\tau_i$ are efficiently separated by $Z$.}
Now we apply Carmesin's theorem as a black box in all the modified torsos $H$ of~$T$.
That is, for every modified torso $H$ of $T$ we obtain a tree set $T_H$ of finite-order separations of $H$ that efficiently distinguishes all the ends of $H$.
Finally, we lift all of Carmesin's tree sets compatibly with each other and with $T$ to obtain a tree set $T'$ of finite-order separations of $G$ that extends $T$.
In the end, every separation in $T_H$ which efficiently distinguishes two ends $\eta_i$ in $H$, with the $\eta_i$ as in the assertion above, gets lifted to a separation in $T'$ that efficiently distinguishes the $\tau_i$ in $G$.

\begin{figure}
    \centering
    \begin{tikzpicture}[decoration=snake,thick,dot/.style={colourdot,circle,fill,inner sep=0.04cm,minimum size=0.1cm},labelgiven/.style={label distance=0cm,inner sep=0.08cm}]
\pgfdecorationsegmentamplitude=1.5pt
\pgfdecorationsegmentlength=0.2cm
\newcommand{\Zwidth}{0.4cm}
\newcommand{\Hwidth}{3cm}
\newcommand{\graphheight}{1cm}
\newcommand{\outerxshift}{2.5cm}
\newcommand{\outeryshift}{0.3cm}
\newcommand{\bendindex}{.3*\outerxshift}
\newcommand{\outeryradius}{2.5cm}
\newcommand{\labelheight}{0.2cm}
\newcommand{\Zangle}{40}
\newcommand{\Gangle}{70}
\draw 
	(0,-\graphheight)
	.. controls +(\bendindex,0) and +(-\bendindex,0)..
	++(\outerxshift,-\outeryshift)
	arc [start angle=-90, end angle=90,x radius=\outeryradius+\outeryshift, y radius=\graphheight+\outeryshift]
	node [pos=0.7,above] {$G$}
	.. controls +(-\bendindex,0) and +(\bendindex,0)..
	++(-\outerxshift,-\outeryshift)
	.. controls +(-\bendindex,0) and +(\bendindex,0) ..
	++(-\outerxshift,\outeryshift)
	arc [start angle=90, end angle=270,x radius=\outeryradius + \outeryshift,y radius=\graphheight+\outeryshift]
	.. controls +(\bendindex,0) and +(-\bendindex,0) ..
	++(\outerxshift,\outeryshift);
\draw[pictureorange] 
	(0,0) ellipse [x radius=\Hwidth,y radius=\graphheight];
\begin{onlayer}{filllayer}
\fill [pictureorange!10]
	(0,0) ellipse [x radius=\Hwidth,y radius=\graphheight];
\end{onlayer}
\draw[red] 
(0,0) ellipse [x radius=\Zwidth,y radius=\graphheight];
\begin{onlayer}{filllayer}
\fill [red!10]
(0,0) ellipse [x radius=\Zwidth,y radius=\graphheight];
\end{onlayer}

\draw[colourray,decorate,->]
	(0,0.8*\graphheight) ++(0cm,0cm) -- (-0.4*\Hwidth,0.6*\graphheight) node [pos=0.9,label=below:$\eta_1$] {};
\draw[colourray,decorate,->]
	(\Zwidth+0.28cm,0cm) -- node [pos=0.4,label=below:$\eta_2$] {} (\Hwidth-0.5cm,0cm);
\path
	(\Zwidth,0cm) node [red,anchor=east,labelgiven] {$Z$}
	(canvas polar cs:angle=45,x radius=\Hwidth, y radius=\graphheight) node [pictureorange,labelgiven,anchor=north east] {$H$};
\path
    ( \Hwidth+1cm,0cm) node [anchor=mid] {$\tau_2$}
    (-\Hwidth-1cm,0cm) node [anchor=mid] {$\tau_1$};
\end{tikzpicture}
    \caption{The separator $Z$ of a separation efficiently distinguishing two inequivalent tangles $\tau_1$ and $\tau_2$; and a modified torso $H$ with ends $\eta_1$ and $\eta_2$ representing the two tangles.}
    \label{fig:simpleReflection}
\end{figure}
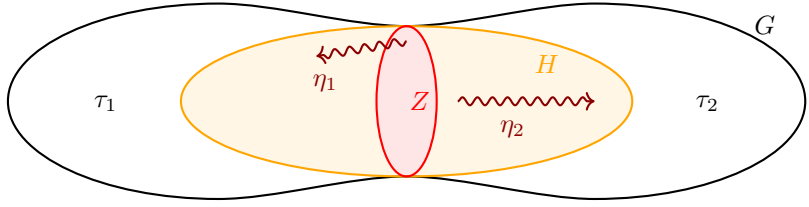

Phrased differently, we reflect the problem of efficiently distinguishing two inequivalent $\aleph_0$-tangles down to the modified torsos of $T$.
There, the problem reduces to efficiently distinguishing two proxy ends, a problem that has already been solved by Carmesin.
Finally, we lift the solutions for the modified torsos of $T$ up to the original graph $G$ to solve the original problem.

\section{From principal collections of separators to tree sets}\label{sec:PrincipalVertexSetsToTreeSets}

\noindent In this section, we show how the separations $\csep{X}$ can be slightly modified to give rise to a tree set that comes with quite a list of useful properties.
Even though our initial intention is to consider these separations for critical vertex sets $X$ of $G$, we can prove a much stronger result by more generally considering what we call \principal\ collections of vertex sets:

\begin{definition}\label{DefinitionPrincipalCollection}
Given a collection $\cY$ of vertex sets of $G$ we say that a vertex set $X$ of $G$ is $\cY$-\emph{\principal } if $X$ meets for every $Y\in\cY$ at most one component of $G-Y$.
And we say that $\cY$ is \emph{\principal } if all its elements are $\cY$-\principal .
\end{definition}

\begin{notation}
If $X\subset V(G)$ meets precisely one component of $G-Y$ for some $Y\subset V(G)$, then we denote this component by $C_Y(X)$.
\end{notation}

\begin{definition}
A set $X\in\cX$ is \emph{\principal } if it is $\cX$-\principal .
\end{definition}

\begin{example}
An $X\in\cX$ is \principal , e.g., if it induces a clique $G[X]$ or is included in a critical vertex set of $G$.
\end{example}

Since \principal\ vertex sets behave like cliques it is possible to alter the graph $G$ so that all \principal\ vertex sets actually induce cliques while the finite-order separations stay the same:

\begin{lemma}\label{principalCliques}
Suppose that $\cY$ is a collection of \principal\ vertex sets of $G$ and let $G_{\cY}$ be obtained from $G$ by turning each $G[X]$ with $X\in\cY$ into a clique.
Then the finite-order separations of $G$ are precisely the finite-order separations of $G_{\cY}$.
In particular, $\Theta(G)=\Theta(G_{\cY})$.
\end{lemma}
\begin{proof}
If $\{A,B\}$ is a finite-order separation of $G$, then each principal $X\in\cY$ meets at most one component of $G-(A\cap B)$.
Therefore, no $X$ adds an $(A\setminus B)$--$(B\setminus A)$ edge in $G_{\cY}$, so $\{A,B\}$ is also a finite-order separation of $G_{\cY}$.
The converse holds due to $E(G_{\cY})\supset E(G)$.
\end{proof}

We will use this lemma in Section~\ref{sec:MainProof} to assume without loss of generality that, for a certain tree set, the torsos coincide with the parts.
Our next definition extends `forming a problem case' from critical vertex sets to arbitrary vertex sets:

\begin{definition}
Two vertex sets $X$ and $Y$ of $G$ with $\{X,Y\}$ \principal\ are said to \emph{form a problem case} if $X$ and $Y$ are incomparable as sets and additionally $C_X(Y)\in\CX$ and $C_Y(X)\in\CY$ hold.
\end{definition}

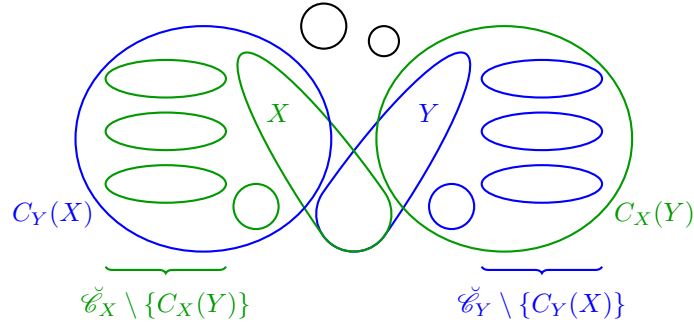
\begin{figure}[ht]
    \centering
    	\begin{tikzpicture}[thick]
	\draw [rotate=55,blue] (0cm,-0.5cm) arc [x radius=2.6cm,y radius=0.5cm,start angle=-90, end angle=90] arc [radius=0.5cm,start angle=90, end angle=270];
	\draw (1cm,1.3cm) node [anchor=mid,blue] {$Y$};
	\draw [rotate=125,picturegreen] (0cm,-0.5cm) arc [x radius=2.6cm,y radius=0.5cm,start angle=-90, end angle=90] arc [radius=0.5cm,start angle=90, end angle=270];
	\draw [picturegreen] (2cm,1cm) ellipse [x radius=1.7cm,y radius=1.5cm];
	\draw (4cm,0cm) node [picturegreen] {$C_X(Y)$};
	\foreach \i in {1,2,3}
	\draw [blue] (2.5cm,-0.3cm+0.7*\i cm) ellipse [x radius=0.8cm, y radius=0.25cm];
	\draw [blue](1.3cm,0.1cm) ellipse [radius=0.3cm];
	\draw [decorate, decoration={brace,mirror},blue](1.7cm,-0.7cm) -- node [label=below:$\breve{\mathscr{C}}_Y\setminus \{C_Y(X)\}$] {} ++(1.6cm,0cm);
	
	\draw [blue] (-2cm,1cm) ellipse [x radius=1.7cm,y radius=1.5cm];
	\foreach \i in {1,2,3}
	\draw [picturegreen] (-2.5cm,-0.3cm+0.7*\i cm) ellipse [x radius=0.8cm, y radius=0.25cm];
	\draw [picturegreen](-1.3cm,0.1cm) ellipse [radius=0.3cm];
	\draw [decorate, decoration={brace},picturegreen](-1.7cm,-0.7cm) -- node [label=below:$\breve{\mathscr{C}}_X\setminus \{C_X(Y)\}$] {} ++(-1.6cm,0cm);
	\draw (-4cm,0cm) node [blue] {$C_Y(X)$};
	\draw (-1cm,1.3cm) node [picturegreen,anchor=mid] {$X$};
	\draw (0.4cm, 2.3cm) ellipse [radius=0.2cm];
	\draw (-0.4cm, 2.5cm) ellipse [radius=0.3cm];
	\end{tikzpicture}
    \caption{Two incomparable sets $X$ and $Y$ such that $\{X,Y\}$ is principal. Note that every component of $G-X$ which is neither $C_X(Y)$ nor contained in $C_Y(X)$ has its neighbourhood in $X\cap Y$ and is thus also a component of $G-Y$ (black circles).
    Also, not every component of $G-Y$ which is contained in $C_X(Y)$ has to be contained in $\CY$, as is depicted by the blue circle on the right.
    If $C_Y(X)\notin \cD \subset \CY$, then $\cD$ is a subset of $\CY\setminus \{C_Y(X)\}$ and thus $\rsep{X}{C_X(Y)}\leq \rsep{Y}{\cD}$.}
    \label{fig:problemcase}
\end{figure}

The following lemma will keep proofs short:

\begin{lemma}\label{SeparationsAreSimple}
If $\sep{X}{\cC}$ and $\sep{Y}{\cD}$ are separations of $G$ satisfying $X\cup V[\cC]\supset Y\cup V[\cD]$ and that each component in $\cD$ avoids $X$, then $\rsep{X}{\cC}\le\rsep{Y}{\cD}$.
\end{lemma}
\begin{proof}
It remains to show $V\setminus V[\cC]\subset V\setminus V[\cD]$ which is tantamount to $V[\cD]\subset V[\cC]$, which in turn is evident from the assumptions.
\end{proof}

In the previous section, we have seen that for two distinct critical (in particular principal) vertex sets $X\neq Y$ their separations $\csep{X}$ and $\csep{Y}$ need not be nested.
This may happen, for example, if $X$ and $Y$ form a problem case.
The following two lemmas show that actually this may happen only if $X$ and $Y$ form a problem case.

\begin{lemma}\label{CriticalSetsComparable}
Let $X\subsetneq Y$ be two vertex sets of $G$ such that $\{X,Y\}$ is \principal .
Then all of the components in $\CY$ are properly contained in the component $C_X(Y)$.
Notably, $C_X(Y)\in\CX$ if $\CY$ is non-empty.
Moreover, if we are given subsets $\cC\subset\CX$ and $\cD\subset\CY$, then
\begin{align*}
    \vs\,\le\rsep{X}{C_X(Y)}\le\crsep{Y}\le\rsep{Y}{\cD}\text{ where }\left\{\begin{array}{cl}
         \vs\,=\rsep{X}{\cC}  & \text{if }C_X(Y)\in\cC\\
         \vs\,=\lsep{\cC}{X} & \text{otherwise}
    \end{array}\right.
\end{align*}
so in particular $\sep{X}{\cC}$ and $\sep{Y}{\cD}$ are nested with each other.
If additionally the inclusion $\cC\subset\cC_X$ is proper, i.e., if $\rsep{X}{\cC}$ is not small, then $\rsep{X}{\cC}\not\le\lsep{\cD}{Y}$.
\end{lemma}
\begin{proof}
Since every component $C\in\CY$ has neighbourhood precisely equal to $Y$, it follows from $X\subsetneq Y$ that $\bigcup\,(\CY\rest X)\subsetneq C_X(Y)$.
Hence Lemma~\ref{SeparationsAreSimple} yields $\rsep{X}{C_X(Y)}\le\crsep{Y}$.

It remains to show $\rsep{X}{\cC}\not\le\lsep{\cD}{Y}$ given that $\cC\subsetneq\cC_X$.
Assume for a contradiction that the inequality $\rsep{X}{\cC}\le\lsep{\cD}{Y}$ holds.
If $C_X(Y)$ is contained in~$\cC$, then $\vs=\rsep{X}{\cC}\le\rsep{Y}{\cD}$ is equivalent to $\lsep{\cD}{Y}\le\lsep{\cC}{X}$ which yields $\rsep{X}{\cC}\le\lsep{\cD}{Y}\le\lsep{\cC}{X}$.
Hence $\rsep{X}{\cC}$ is small which implies $\cC=\cC_X$, a~contradiction.
Otherwise $C_X(Y)$ is not contained in~$\cC$, and then $\vs=\lsep{\cC}{X}\le\rsep{Y}{\cD}$ is equivalent to $\lsep{\cD}{Y}\le\lsep{X}{\cC}$ which yields $\rsep{X}{\cC}\le\lsep{\cD}{Y}\le\lsep{X}{\cC}$.
In particular, $\rsep{X}{\cC}=\lsep{\cD}{Y}$ implies $X=Y$, contradicting the fact that~$X\subsetneq Y$.
\end{proof}

Our next lemma is also illustrated in Figure~\ref{fig:problemcase}.

\begin{lemma}\label{CriticalSetsIncomparableNoProblem}
Let $X$ and $Y$ be two incomparable vertex sets of $G$ such that $\{X,Y\}$ is \principal .
If we are given subsets $\cC\subset\CX$ and $\cD\subset\CY$ with $C_Y(X)\notin\cD$, then
\begin{align*}
    \vs\,\le\rsep{X}{C_X(Y)}\le\rsep{Y}{\cD}\text{ where }\left\{\begin{array}{cl}
         \vs\,=\rsep{X}{\cC} & \text{if }C_X(Y)\in\cC\\
         \vs\,=\lsep{\cC}{X} & \text{otherwise}
    \end{array}\right.
\end{align*}
so in particular $\sep{X}{\cC}$ and $\sep{Y}{\cD}$ are nested with each other and we have $\rsep{X}{\cC}\not\le\lsep{\cD}{Y}$.
\end{lemma}
\begin{proof}
The assumption $C_Y(X)\notin\cD$ ensures that every component in $\cD$ avoids $X$.
Let $y$ be any vertex in $Y\setminus X$.
As every component in $\cD$ avoids $X$ and sends an edge to $y\in Y\setminus X$, we deduce that $(Y\setminus X)\cup\bigcup\cD\subset C_X(Y)$.
Hence Lemma~\ref{SeparationsAreSimple} yields $\rsep{X}{C_X(Y)}\le\rsep{Y}{\cD}$.
From this, the rest is evident.
\end{proof}

We are now ready to prove our main technical result, Theorem~\ref{TheTreeSetAPC}.
To allow for more flexibility in its applications, we have extracted the following definition and second main technical result from Theorem~\ref{TheTreeSetAPC}:

\begin{definition}
Suppose that $\cY$ is a \principal\ collection of vertex sets of $G$.
A~function that assigns to every $X\in\cY$ a subset $\ccK(X)\subset\CX$ is called \emph{\admissable } for $\cY$ if for every two $X,Y\in\cY$ that are incomparable as sets we have either $C_X(Y)\notin\ccK(X)$ or $C_Y(X)\notin\ccK(Y)$.
If additionally $\vert\CX\setminus\ccK(X)\vert\le 1$ for all $X\in\cY$, then $\ccK$ is \emph{strongly} \admissable\ for $\cY$.
\end{definition}

\begin{theorem}\label{criticalAdmissableExists}
For every \principal\ collection of vertex sets of a connected graph there is a strongly \admissable\ function.
\end{theorem}

\begin{proof}
Let $\cY$ be a \principal\ collection of vertex sets of a connected graph $G$.
We write $\cP$ for the collection of those \principal\ vertex sets in $\cY$ that form a problem case with some other \principal\ vertex set in $\cY$.
Let us fix any well-ordering of $\cP$ and view $\cP$ as well-ordered set from now on.

For each $X\in\cP$ we put $K(X):=C_X(Y)$ for the first $Y\in\cP$ which forms a problem case with $X$.
Let us put $\ccK(X):=\CX\setminus\{K(X)\}$ for every $X\in\cP$, and $\ccK(X):=\CX$ for all other vertex sets $X\in\cY$.
We claim that $\ccK$ is strongly \admissable\ for $\cY$.

For this, let $X\neq Y$ be any two distinct vertex sets in $\cY$ that form a problem case.
We show that at least one of $K(X)=C_X(Y)$ and $K(Y)=C_Y(X)$ holds.
Let $Z\in\cP$ be the first vertex set that forms a problem case with one of $X$ and $Y$.
Without loss of generality we may assume that $Z$ forms a problem case with $X$, so we have $K(X)=C_X(Z)$ by the minimal choice of $Z$.
Since we are done if $Y$ and $Z$ meet the same component of $G-X$, we may assume that $C_X(Y)\neq C_X(Z)$.
This means that the three sets $X,Y,Z$ are pairwise incomparable.
Our plan is to show that $Y$ forms a problem case with $Z$, and that this gives $K(Y)=C_Y(Z)=C_Y(X)$ as desired.

We already know that $Y$ and $Z$ are incomparable.
Next, let us verify that $C_Y(Z)\in\CY$.
For this, pick any vertex $x\in X\setminus Y$.
As $X$ and $Z$ form a problem case we have $C_X(Z)\in\CX$, so the vertex $x$ sends some edge $e$ to the component $C_X(Z)$.
Now $x$ is not in $Y$ and the component $C_X(Z)$ avoids $Y$ as $Y$ and $Z$ live in distinct components of $G-X$ by assumption, so $C_X(Z)+e$ is a connected subgraph of $G-Y$ that meets both $X$ and $Z$, yielding $C_Y(Z)=C_Y(X)$.
Since $Y$ and $X$ form a problem case, giving $C_Y(X)\in\CY$, we get $C_Y(Z)\in\CY$ as required.
By symmetry we have $C_Z(Y)\in\CZ$, so $Y$ and $Z$ form a problem case as desired and $K(Y)=C_Y(Z)$ follows from the minimal choice of $Z$. 
To see that $K(Y)=C_Y(X)$ holds, recall that we proved $C_Y(Z)=C_Y(X)$ three sentences earlier.
\end{proof}

Finally, we go for Theorem~\ref{TheTreeSetAPC}, which considers tree sets of the following form:

\begin{notation}
Given a \principal\ collection $\cY$ of vertex sets of $G$ and an \admissable\ function $\ccK$ for $\cY$ we write
\begin{align*}
    T(\cY,\ccK):=\big\{\,\tsep{X}\,,\,\sep{X}{K}\;\big\vert\;X\in\cY\text{ and }K\in\ccK(X)\,\big\}.
\end{align*}
For every vertex set $X\in\cY$ we write $\sigma_X^{\ccK}$ for the star that consists of the separation $\trsep{X}$ and all the separations $\lsep{K}{X}$ with $K\in\ccK(X)$.
Notably, each star $\sigma_X^{\ccK}$ has interior $X$.
\end{notation}

\begin{figure}
    \centering
    \begin{tikzpicture}[thick, decoration=brace,
    drawfill/.style={draw=#1,fill=#1!20},
    local colour/.style={color=#1, drawfill/.default=#1},
    local colour=black,
    sepname/.store in=\sepname,
    ellipse length/.store in=\ellipselength,
    corner angle/.store in=\cornerangle,
    dot/.style={inner sep=2pt, circle, fill},
    dots/.pic={\foreach \i in {-1,0,1} \draw (0cm, \i*0.1 cm) node [shape=circle, fill,inner sep=0.3pt,anchor=center] {};}]
\newcommand{\cornerlength}{0.4cm}
\newcommand{\cornerangle}{60}
\newcommand{\rotationangle}{30}
\newcommand{\separatorheight}{0.8cm}

\draw (-2cm,\separatorheight) ..controls +(0:1.5cm) and +(180+\rotationangle:0.5cm) .. ($(1cm,1cm)+(90+\rotationangle:\separatorheight)$);
\draw (-2cm,-\separatorheight) .. controls +(0:1.5cm) and +(180-\rotationangle:0.5cm) .. ($(1cm,-1cm)+(-90-\rotationangle:\separatorheight)$);
\draw ($(1cm,1cm)+(-90+\rotationangle:\separatorheight)$) .. controls +(180+\rotationangle:0.3cm) and +(180-\rotationangle:0.3cm).. ($(1cm,-1cm) + (90-\rotationangle:\separatorheight)$);

\begin{scope}[shift={(1cm,1.0cm)}, rotate=\rotationangle,local colour=pictureorange, sepname={X},ellipse length=1.0cm]
\draw [black](0cm,-\separatorheight)
    ..controls +(0:0.5cm) and +(180:0.5cm) ..
    (0.5cm+\ellipselength,-\separatorheight-0.1cm)
    .. controls +(0:0.5cm)  and +(\cornerangle-180:\cornerlength) .. 
    ($(0.5cm+2*\ellipselength,0cm)+ (0.2cm,-0.5cm)$) 
    .. controls +(\cornerangle:\cornerlength) and +(-\cornerangle:\cornerlength) ..
    ($(0.5cm+2*\ellipselength,0cm)+ (0.2cm,+0.5cm)$) 
    .. controls +(180-\cornerangle:\cornerlength) and +(0:0.5cm) .. 
    (0.5cm+\ellipselength, \separatorheight+0.1cm)
    .. controls +(180:0.5cm) and +(0:0.5cm)..
    (0cm,\separatorheight);
\path [drawfill] (0,0) ellipse [x radius=0.25 cm, y radius=\separatorheight];
\draw [->] (0cm,0cm) -- (-0.6cm, 0cm);
\draw  (0cm,\separatorheight+0.2cm) node [anchor=south,inner sep=0pt] {$\sepname$};
\foreach \i in {0.6,0.2,-0.6}
\draw (0.5cm+\ellipselength, \i cm)  ellipse [x radius=\ellipselength, y radius=0.15 cm];
\draw (0.5cm+\ellipselength,-0.2cm) pic [transform shape] {dots};
\draw [decorate] (0.5cm+2*\ellipselength,-\separatorheight-0.15cm) -- node[pos=0.6, anchor=north west,inner sep=1pt] {$\ccK(\sepname)$} ++(-2*\ellipselength,0cm);
\end{scope}

\begin{scope}[shift={(1cm,-1.0cm)}, rotate=-\rotationangle,local colour=picturegreen, sepname={Y},ellipse length=1.0cm]
\draw [black](0cm,-\separatorheight)
    ..controls +(0:0.5cm) and +(180:0.5cm) ..
    (0.5cm+\ellipselength,-\separatorheight-0.1cm)
    .. controls +(0:0.5cm)  and +(\cornerangle-180:\cornerlength) .. 
    ($(0.5cm+2*\ellipselength,0cm)+ (0.2cm,-0.5cm)$) 
    .. controls +(\cornerangle:\cornerlength) and +(-\cornerangle:\cornerlength) ..
    ($(0.5cm+2*\ellipselength,0cm)+ (0.2cm,+0.5cm)$) 
    .. controls +(180-\cornerangle:\cornerlength) and +(0:0.5cm) .. 
    (0.5cm+\ellipselength, \separatorheight+0.1cm)
    .. controls +(180:0.5cm) and +(0:0.5cm)..
    (0cm,\separatorheight);
\path [drawfill] (0,0) ellipse [x radius=0.25 cm, y radius=\separatorheight];
\draw [->] (0cm,0cm) -- (-0.6cm, 0cm);
\draw  (0cm,-\separatorheight-0.2cm) node [anchor=north,inner sep=0pt] {$\sepname$};
\foreach \i in {0.6,0.2,-0.6}
\draw (0.5cm+\ellipselength, \i cm)  ellipse [x radius=\ellipselength, y radius=0.15 cm];
\draw (0.5cm+\ellipselength,-0.2cm) pic [transform shape] {dots};
\draw [decorate] (0.5cm+2*\ellipselength,-\separatorheight-0.15cm) -- node[pos=0.35, anchor=north east,inner sep=2pt] {$\ccK(\sepname)$} ++(-2*\ellipselength,0cm);
\end{scope}

\begin{scope}[xscale=-1,shift={(2cm,0cm)},local colour=picturedarkred, sepname={Z},ellipse length=1.5cm, corner angle=55]
\draw [black](0cm,-\separatorheight)
    ..controls +(0:0.5cm) and +(180:0.5cm) ..
    (0.5cm+1.2*\ellipselength,-\separatorheight-0.1cm)
    .. controls +(0:0.5cm)  and +(\cornerangle-180:\cornerlength) .. 
    ($(0.5cm+2*\ellipselength,0cm)+ (0.3cm,-0.5cm)$)
    .. controls +(\cornerangle:\cornerlength) and +(-\cornerangle:\cornerlength) ..
    ($(0.5cm+2*\ellipselength,0cm)+ (0.3cm,+0.5cm)$)
    .. controls +(180-\cornerangle:\cornerlength) and +(0:0.5cm) .. 
    (0.5cm+1.2*\ellipselength, \separatorheight+0.1cm)
    .. controls +(180:0.5cm) and +(0:0.5cm)..
    (0cm,\separatorheight);
\path [drawfill] (0,0) ellipse [x radius=0.25 cm, y radius=\separatorheight];
\draw [->] (0cm,0cm) -- (-0.6cm, 0cm);
\draw  (0cm,-\separatorheight-0.2cm) node [anchor=north, inner sep=0pt] {$\sepname$};
\draw (0.5cm+\ellipselength,0.65 cm) ellipse [x radius=1.5 cm, y radius=0.1 cm];
\draw (0.5cm+\ellipselength,0.35cm) pic {dots};
\draw (0.5cm,-0.75cm) [rounded corners=10pt] rectangle (0.5cm+2*\ellipselength,0.15 cm);
\draw [decorate] (0.5cm+2*\ellipselength,\separatorheight+0.15cm) -- node[pos=0.55, anchor=south, inner sep=3pt] {$\ccK(\sepname)$} ++(-2*\ellipselength,0cm);
\end{scope}

\begin{scope}[xscale=-1, shift={(2cm,0cm)},local colour=blue, sepname={W},ellipse length=0.8cm]
\path [drawfill] (1.5cm,-0.3cm) ellipse [x radius=0.15 cm, y radius=0.45cm];
\path (1.5cm,-\separatorheight-0.05cm) node [anchor=north] {$\sepname$};
\draw [->] (1.5cm,-0.3cm) -- ++(-0.5cm,0cm);
\foreach \i in {0,-0.6}
\draw (2.5cm,\i cm) ellipse [x radius=\ellipselength, y radius=0.1cm];
\draw (2.5cm,-0.3cm) pic {dots};
\draw [decorate] (1.7cm,-\separatorheight-0.15cm) -- node [midway, anchor=north] {$\ccK(\sepname)$} ++(2*\ellipselength,0cm);
\end{scope}

\end{tikzpicture}
    \caption{A principal set $\mathcal{Y}=\{W,X,Y,Z\}$ of pairwise disjoint sets and the separations of the form $\tlsep{X'}$ for $X'\in \mathcal{Y}$ where $\ccK$ is some \admissable\ function for $\mathcal{Y}$. Note that in accordance with part (i) of Theorem~\ref{TheTreeSetAPC} the depicted separations form a partial consistent orientation.}
    \label{fig:consistenorientation2}
\end{figure}
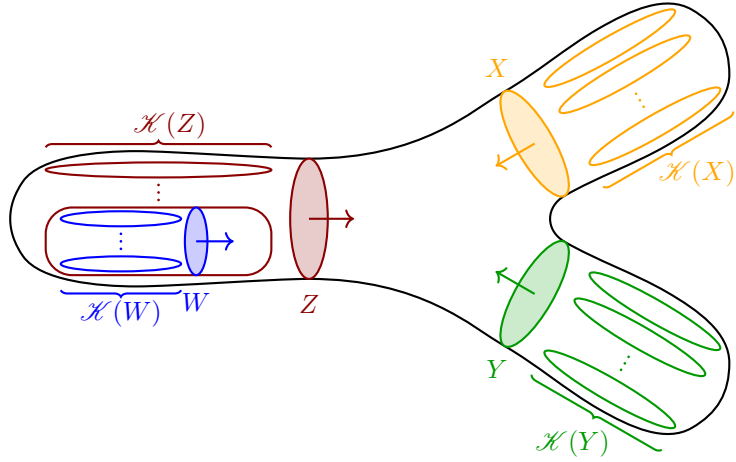

\begin{theorem}\label{TheTreeSetAPC}
Let $G$ be any connected graph, let $\cY$ be a \principal\ collection of vertex sets of $G$ and let $\ccK$ be an \admissable\ function for $\cY$.
Abbreviate $T(\cY,\ccK)=T$ and $\sigma_X^\ccK=\sigma_X$.
Then the following assertions hold:

\begin{enumerate}[topsep=.5em]\itemsep.5em
    \item For every distinct two $X,Y\in\cY$, after possibly swapping $X$ and $Y$, either
    \begin{align*}
        &\tlsep{X}\le\trsep{Y}\\
        \text{or }&\trsep{X}\le\rsep{X}{C_X(Y)}\le\trsep{Y}.
    \end{align*}
    The collection of all separations $\tlsep{X}$ with $\ccK(X)\subsetneq\cC_X$ forms  a consistent partial orientation of $T$.
    \item The collection $T$ of separations is nested.\\
    It is a regular tree set if $\emptyset\subsetneq\ccK(X)\subsetneq\cC_X$ holds for all $X\in\cY$.
    \item Every star $\sigma_X$ with $X\in\cY$ is a splitting star of $\vT$.
\end{enumerate}
Moreover, if all the vertex sets in $\cY$ are finite, then we may add:
\begin{enumerate}[resume,topsep=.5em]\itemsep.5em
    \item If $\tau$ is an ultrafilter tangle of $G$ with $X_\tau\in\cY$ and $\ccK(X_\tau)\in U(\tau,X_\tau)$, then $\tau$ induces via $\tau\mapsto\tau\cap \vT$ on $T$ the consistent orientation which is given by the infinite splitting star $\sigma_{X_\tau}$ in that $\tau\cap \vT=\dc{\sigma_{X_\tau}}$.
    \item If $\crit(G)\subset\cY$ and $\CX\setminus\ccK(X)$ is finite for all $X\in\crit(G)$, then $T$ distinguishes every two inequivalent ultrafilter tangles $\tau_1$ and $\tau_2$ of $G$ via separations in $\sigma_{X_{\tau_1}}$ and $\sigma_{X_{\tau_2}}$, and it distinguishes every end from every ultrafilter tangle $\tau$ via a separation in $\sigma_{X_\tau}$.
\end{enumerate}
\end{theorem}

\begin{proof}
(i) We start with the inequalities.
If $X$ and $Y$ are comparable with $X\subsetneq Y$, say, then we are done by Lemma~\ref{CriticalSetsComparable}.
Otherwise $X$ and $Y$ are incomparable, and then we are done by Lemma~\ref{CriticalSetsIncomparableNoProblem} since $\ccK$ is \admissable .
The two lemmas also prevent $\trsep{X}\le\tlsep{Y}$ for all distinct two elements $X,Y\in\cY$ with $\ccK(X)\subsetneq\cC_X$ and $\ccK(Y)\subsetneq\cC_Y$, ensuring that the partial orientation of~$T$ formed by the separations $\tlsep{X}$ with $\ccK(X)\subsetneq\cC_X$ is consistent.

(ii) That $T$ is nested follows from (i). For the `moreover' part note that requiring $\emptyset\subsetneq\ccK(X)\subsetneq\cC_X$ ensures that $\tsep{X}$ has no small orientation.

(iii)
It suffices to show that every separation in $T$ with separator $Y\neq X$ has an orientation that lies below some element of $\sigma_X$.
So consider any $Y\in\cY$ other than $X$.
Since $\sigma_Y$ is a star, it suffices to show that some separation in $(\sigma_Y)^\ast$ lies below some element of $\sigma_X$.
By (i) it suffices to consider the following cases.
If $\tlsep{Y}\le\trsep{X}$ we are done.
Otherwise either 
\begin{align*}
    &\trsep{X}\le\rsep{X}{C_X(Y)}\le\trsep{Y}\\
    \text{or }&\trsep{Y}\le\rsep{Y}{C_Y(X)}\le\trsep{X}.
\end{align*}
In the first case we are fine since $\tlsep{Y}\le\lsep{C_X(Y)}{X}\in\sigma_X$.
And in the second case we are done by the second inequality.

(iv)
Let $\tau$ be any ultrafilter tangle of $G$ with $X_\tau\in\cY$ and write $X=X_\tau$.
First, we show that $\sigma_X$ is included in $O:=\tau\cap \vT$.
The assumption $\ccK(X)\in U(\tau,X)$ means $\trsep{X}\in O$.
Moreover, we have $\lsep{K}{X}\in\tau$ for every $K\in\ccK(X_\tau)$ as $U(\tau,X)$ is a free ultrafilter.
Thus $\sigma_X\subset O$, and so $\dc{\sigma_X}\subset O$ by consistency.
Conversely, $O\subset\dc{\sigma_X}$ since $\sigma_X\subset O$ is a splitting star of $\vT$ by~(iii).

(v)
If $\tau_1$ and $\tau_2$ are two ultrafilter tangles of $G$ with $X_{\tau_1}\neq X_{\tau_2}$, then the induced orientations $\tau_i\cap\vT$ come from distinct splitting stars $\sigma_{X_{\tau_i}}$ of $\vT$ by~(iii).
Now if $\omega$ is an end of $G$ and $\tau$ is an ultrafilter tangle, then $\omega$ avoids the star $\sigma_{X_\tau}$ since it has finite interior (cf.~\cite[Corollary~1.7]{EndsAndTangles}) while $\tau$ contains it by~(iii).
\end{proof}

We close this section by showing that in general it is not possible to find an \admissable\ function $\ccK$ for which $\vT(\crit(G),\ccK)$ is a tree set that is even isomorphic to the edge tree set of a tree.

\begin{figure}[ht]
	\centering
	\includegraphics[scale=.6]{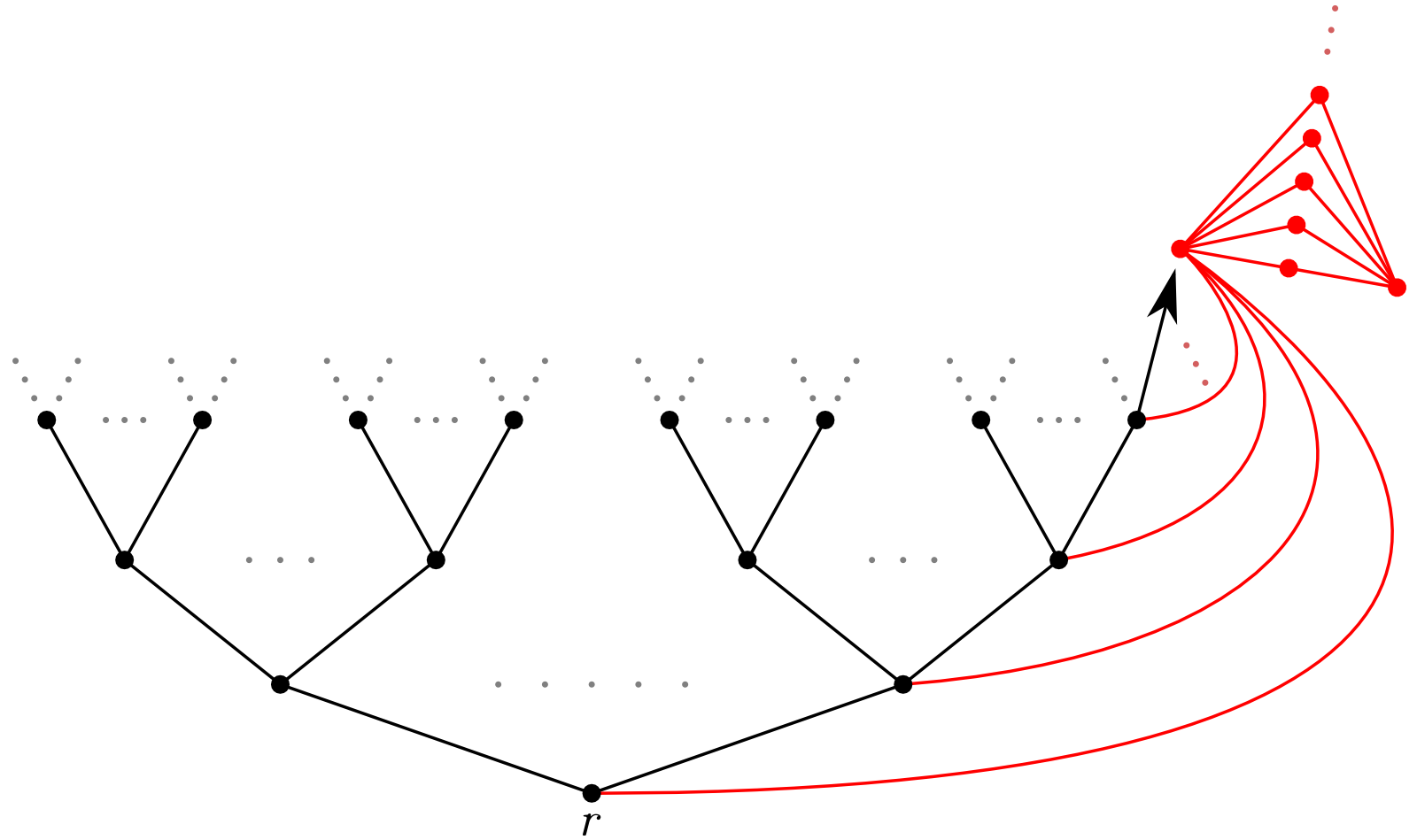}
	\caption{A $T_{\aleph_0}$ (black) with $2^{\aleph_0}$ many copies of $K_{2,\aleph_0}$ as `tops' (visualised in red for the right-most ray).}
	\label{fig:binTree}
\end{figure}

\begin{example}
If $G$ is the graph shown in Figure~\ref{fig:binTree}, then there is no function assigning to each critical vertex set $X$ of $G$ a cofinite subset $\ccK(X)\subset\CX$ such that
\begin{align*}
N:=\big\{\,\tsep{X}\;\big\vert\;X\in\crit(G)\,\big\}
\end{align*}
gives rise to a tree set $\vN$ that is isomorphic to the edge tree set of a tree (so in particular it cannot be induced by an $\Sinf$-tree or tree-decomposition of $G$).
First, however, we describe $G$ more precisely.
The graph $G$ is obtained from the $\aleph_0$-regular tree $T=T_{\aleph_0}$ by fixing any root $r$ and then proceeding as follows.
For every ray $R\subset T$ starting at the root $r$ we add a new copy of $K_{2,\aleph_0}$ with 2-class $\{x_R,y_R\}$, say, and join $x_R$ to every vertex of the ray $R$.
Readers familiar with the `binary tree with tops' will note that $G$ extends a `$T_{\aleph_0}$ with tops'.

Let us check that there really is no suitable function $X\mapsto\ccK(X)$ as claimed.
Assume for a contradiction that there is.
Then $\vN$ is a tree set that, by Theorem~\ref{KneipTreeSets}, has no $(\omega+1)$-chains.
Hence to yield a contradiction, it suffices to find an $(\omega+1)$-chain.
If $t$ is a node of $T\subset G$, then its down-closure $\lceil t\rceil$ in $T$ is a critical vertex set of $G$, and the components in $\CC{\lceil t\rceil}$ are of the following form.
If $t'$ is an upward neighbour of $t$ in $T$, then the vertex set of the component of $G-\lceil t\rceil$ containing $t'$ is given by the union of $\lfloor t'\rfloor\subset T$ with all the copies of $K_{2,\aleph_0}$ whose corresponding ray has a tail in $\lfloor t'\rfloor$.
This gives a bijection between the upward neighbours of $t$ in $T$ and the components in $\CC{\lceil t\rceil}$.
Next, we claim that there is a ray $R^\ast=t_0t_1t_2\ldots\subset T$ starting at the root $r$ such that for all $n>0$ the node $t_n$ corresponds to a component in $\ccK(\lceil t_{n-1}\rceil)$ for its predecessor $t_{n-1}$.
Indeed, since $\ccK(\lceil t\rceil)\subset\CC{\dc{t}}$ is infinite for all $t\in T$, such a ray can be constructed inductively.
But then we get a strictly ascending sequence
\begin{align*}
\trsep{\dc{t_0}}<\trsep{\dc{t_1}}<\trsep{\dc{t_2}}<\cdots
\end{align*}
i.e.\ we get an $\omega$-chain in $\vN$.
And this $\omega$-chain extends to an $(\omega+1)$-chain as the separation $\trsep{Z}$ with $Z=\{x_{R^\ast},y_{R^\ast}\}$ that comes from the $K_{2,\aleph_0}$ for $R^\ast$ is greater than all separations $\trsep{\dc{t_n}}$.\qed
\end{example}

\section{Applications}\label{sec:Applications}

\noindent This section is dedicated to the applications of our work mentioned in the introduction.
All of the four applications are, in fact, applications of Theorems~\ref{criticalAdmissableExists} and~\ref{TheTreeSetAPC}.
Elbracht, Kneip and Teegen~\cite{InfiniteSplinters} use the following corollary of our two theorems:

\begin{corollary}
Suppose that $\cY$ is a {\principal } \footnote{cf.~Definition~\ref{DefinitionPrincipalCollection}} collection of vertex sets of $G$.
Then there is a function $\ccK$ assigning to each $X\in\cY$ a subset $\ccK(X)\subset\CX$ (the set $\CX$ consists of the components of $G-X$ whose neighbourhoods are precisely equal to~$X$) that misses at most one component from $\CX$, such that the collection
\begin{align*}
    \big\{\,\{\,V\setminus K\,,\,X\cup K\,\}\;\big\vert\;X\in\cY\text{ and }K\in\ccK(X)\,\big\}
\end{align*}
is nested.\qed\
\end{corollary}

Bürger and the second author~\cite{StarComb4TheUndominatingStar} use Theorems~\ref{criticalAdmissableExists} and~\ref{TheTreeSetAPC} directly.
In the remainder of this section, we present applications three and four: a structural connectivity result for infinite graphs, and the collectionwise normality of ultrafilter tangle spaces.

\subsection{A structural connectivity result for infinite graphs}

We have already explained this application in detail in our introduction, now we prove it:

\begin{customthm}{\ref{toughCriticalTreeSet}}
Every connected graph $G$ has a tree set whose separators are precisely the critical vertex sets of $G$ and all whose torsos are tough.
\end{customthm}
\begin{proof}
By Theorems~\ref{criticalAdmissableExists} and~\ref{TheTreeSetAPC} it suffices to show that for $\cY:=\crit(G)$ and a stronlgy \admissable\ function $\ccK$ the torsos of the tree set $T(\cY,\ccK)$ are tough.
For this, let $O$ be any consistent orientation of $T(\cY,\ccK)$, let $\Pi$ be its part and $H$ its torso.
In order to show that $H$ is tough, let $\Xi$ be a finite subset of~$V(H)$.
Let $\cC\subset\cC_\Xi$ consist of those components of $G-\Xi$ that meet $\Pi$.
Then $\cC$ must be finite: otherwise $\Xi$ contains a critical vertex set $\Xi'$ of $G$ with $\cC':=\CC{\Xi'}\cap\cC$ infinite; and then $\rsep{\Xi'}{C}\in O$ for all $C\in\cC'\cap\ccK(\Xi')$ as these $C$ meet $\Pi$, contradicting the consistency of $O$.
Thus $G-\Xi$ has only finitely many components meeting $\Pi$.
By Corollary~\ref{restrictconnectedsets} each of these components induces a component of $H-\Xi$, and so $H-\Xi$ has only finitely many components.
\end{proof}

\subsection{Collectionwise normality of the ultrafilter tangle space}
For this subsection, we recall the following definitions from general topology (cf.~\cite{EngelkingBook}):

\begin{definition}[Normality and collectionwise normality]\label{DefinitionCollectionwiseNormal}
Let $X$ be a topological space in which all singletons are closed.

The space $X$ is said to be \emph{normal} if for every two disjoint closed subsets $A_1$ and $A_2$ of $X$ there are disjoint open subsets $O_1$ and $O_2$ of $X$ with $A_1\subset O_1$ and $A_2\subset O_2$.

A collection $\{\,A_i\mid i\in I\,\}$ of subsets $A_i\subset X$ is \emph{discrete} if every point $x\in X$ has an open neighbourhood that meets at most one of the~$A_i$.

The space $X$ is said to be \emph{collectionwise normal} if for every discrete collection $\{\,A_i\mid i\in I\,\}$ of pairwise disjoint closed subsets $A_i\subset X$ there exists a collection $\{\,O_i\mid i\in I\,\}$ of pairwise disjoint open subsets $O_i\subset X$ with $A_i\subset O_i$ for all $i\in I$.
\end{definition}

The following implications are true for every topological space (the first implication is~\cite[Theorems~5.1.1 and 5.1.18]{EngelkingBook} whereas the second is clear):
\begin{align*}
    \text{compact \HD\ }\Rightarrow\text{ collectionwise normal }\Rightarrow\text{ normal.}
\end{align*}

In general, the end space $\Omega(G)$ of a graph is not compact, e.g., if $G$ is a union of infinitely many rays sharing precisely their first vertex.
Polat~\cite{PolatEME1} and Sprüssel~\cite{NormalEnd} independently showed that the end space of every graph $G$ is normal, and Polat even showed that the end space is collectionwise normal (this is Lemma~4.14 in~\cite{PolatEME1}, see~\cite{ApproximatingNormalTrees} for a modern proof):

\begin{theoremNN}
Every graph $G$ has a collectionwise normal end space $\Omega(G)$.
\end{theoremNN}

The $\aleph_0$-tangle space $\Theta(G)$, with the subspace topology imposed by Diestel's tangle \comp , is compact \HD\ and contains the end space as a subspace.
Since every compact \HD\ space is collectionwise normal, the \mbox{$\aleph_0$-tangle} space is collectionwise normal as well:

\begin{theoremNN}
Every graph $G$ has a collectionwise normal $\aleph_0$-tangle space~$\Theta(G)$.
\end{theoremNN}

\noindent This result, however, does not imply that the end space is collectionwise normal, for usually the end space is not closed in the $\aleph_0$-tangle space.

As the $\aleph_0$-tangle space is the disjoint union $\Theta(G)=\Omega(G)\sqcup\Upsilon(G)$ of the end space $\Omega(G)$ and the ultrafilter tangle space $\Upsilon(G)$, the question arises whether the ultrafilter tangle space is collectionwise normal as well.
Like the end space, the ultrafilter tangle space usually is not closed in the $\aleph_0$-tangle space, so the ultrafilter tangle space does not obviously inherit the collectionwise normality from the \mbox{$\aleph_0$-tangle} space.

In this section, we show that the ultrafilter tangle space is collectionwise normal, Theorem~\ref{CollectionWiseNormality}. 
We remark that our proof also shows that the critical vertex set space (with the subspace topology from the \comp\ $\minG=G\sqcup\crit(G)\sqcup\Omega(G)$ introduced in~\cite{EndsTanglesCrit}) is collectionwise normal as well.


\begin{customthm}{\ref{CollectionWiseNormality}}
Let $G$ be any connected graph.
Then for every discrete collection $\{\,A_i\mid i\in I\,\}$ of pairwise disjoint closed subsets $A_i\subset\Upsilon(G)$ there exists a collection $\{\,O_i\mid i\in I\,\}$ of pairwise disjoint open subsets $O_i\subset\modG$ such that $A_i\subset O_i$ for all $i\in I$.
In particular, the ultrafilter tangle space of $G$ is collectionwise normal.
    
    
    
\end{customthm}


Our proof of Theorem~\ref{CollectionWiseNormality} will employ the following short lemma:

\begin{lemma}\label{tcOpenProp}
For every two finite-order separations $\rsep{X}{\cC}\le\rsep{Y}{\cD}$ of $G$ we have $\cO_{\modG}(X,\cC)\supset\cO_{\modG}(Y,\cD)$.
\end{lemma}
\begin{proof}
Clearly, $G\cap\cO_{\modG}(X,\cC)\supset G\cap\cO_{\modG}(Y,\cD)$.
And from the consistency of $\aleph_0$-tangles we deduce $\Theta\cap\cO_{\modG}(X,\cC)\supset\Theta\cap\cO_{\modG}(Y,\cD)$.
\end{proof}

\begin{proof}[Proof of Theorem~\ref{CollectionWiseNormality}]
For this, let $\{\,A_i\mid i\in I\,\}$ by any discrete collection of closed subsets $A_i\subset\Upsilon(G)$.
We are going to find a suitable collection $\{\,O_i\mid i\in I\,\}$.
To get started, we view the $\aleph_0$-tangle space as inverse limit $\Theta=\invLim{}(\,\beta (\cC_X)\mid X\in\cX\,)$ where each $\cC_X$ is endowed with the discrete topology.
Since~$\Theta$ is compact and all $\beta (\cC_X)$ are \HD , it follows from general topology that all of the (continuous) projections $\pr_Y\colon\Theta=\invLim{}\beta (\cC_X)\to\beta(\cC_Y)$ are open.
Now consider any critical vertex set $X$ of $G$.
The \SC\ remainder $(\CX)^\ast=\beta(\CX)\setminus\CX$ is a closed subspace of $\beta (\CX)=\closureInExt{\CX}{\beta(\cC_X)}\subset\beta (\cC_X)$.
(This follows from general topology, but it can also be seen more directly by considering the standard basis for the \SC\ \comp\ of discrete spaces.)
And for every $U\in (\CX)^\ast$ the preimage $\pr_X^{-1}(U)$ is a singleton that consists precisely of the ultrafilter tangle of which $(X,U)$ is the blueprint.
Therefore, for every $i\in I$ the set 
\begin{align*}
    A_{i,X}:=\pr_X(A_i)\cap (\CX)^\ast=\pr_X\big(\;\closureIn{A_i}{\Theta}\,\big)\cap (\CX)^\ast
\end{align*}
is closed in $\beta(\CX)$.
Moreover, $\{\,A_{i,X}\mid i\in I\,\}$ is a discrete collection of pairwise disjoint closed subsets of~$\beta (\CX)$.
Now the \SC\ \comp\ $\beta (\CX)$ is collectionwise normal since it is compact \HD , and so we find a collection $\{\,O_{i,X}\mid i\in I\,\}$ of pairwise disjoint open subsets $O_{i,X}\subset \beta(\CX)$ satisfying the inclusion $A_{i,X}\subset O_{i,X}$ for all $i\in I$.

Next, we use Theorem~\ref{criticalAdmissableExists} to find a strongly \admissable\ function $\ccK$ for the \principal\ collection $\crit(G)$ with $\vert\,\CX\setminus\ccK(X)\,\vert=1$ for all $X\in\crit(G)$.
For every index $i\in I$ and every ultrafilter tangle $\tau\in A_i$ we choose a component collection $\cC(\tau)\in U(\tau,X_\tau)$ such that
\begin{itemize}
    \item $\cC(\tau)\subset\ccK(X_\tau)$;
    \item $\cC(\tau)\subset O_{i,X_\tau}$;
    \item $O_{i,\tau}:=\cO_{\modG}(X_\tau,\cC(\tau))$ avoids all $A_j$ with $j\neq i$.
\end{itemize}
We find $\cC(\tau)$ as follows.
First, we recall that $\ccK(X_\tau)$ is contained in the free ultrafilter $U(\tau,X_\tau)$.
Second, we note that $O_{i,X_\tau}\cap\CC{X_\tau}$ is contained in $U(\tau,X_\tau)$ as well, for $O_{i,X_\tau}$ is an open neighbourhood of $U=\pr_{X_\tau}(\tau)\in A_{i,X_\tau}$ in $\beta (\CC{X_\tau})$ and $U$ is contained in $U(\tau,X_\tau)$ as a subset.
Therefore, if we find a component collection $\cC\subset\CC{X_\tau}$ with $\cC\in U(\tau,X_\tau)$ such that $\cO_{\modG}(X_\tau,\cC)$ avoids all $A_j$ with $j\neq i$, then $\cC(\tau):=\ccK(X_\tau)\cap O_{i,X_\tau}\cap\cC$ will satisfy all three requirements (for the third requirement we apply Lemma~\ref{tcOpenProp} to $\rsep{X_\tau}{\cC}\le\rsep{X_\tau}{\cC(\tau)}$).
To find a suitable component collection~$\cC$, we proceed as follows.
The union of all sets $A_j$ with $j\in I$ and $j\neq i$ is closed in $\Upsilon(G)$ since $\{\,A_i\mid i\in I\,\}$ is a discrete collection of closed sets.
Hence there exists an open neighbourhood $\cO_{\modG}(Y,\cD)$ of $\tau$ in $\modG$ which avoids this union.
Applying Lemma~\ref{ufTangleCofinalSeps} to $\rsep{Y}{\cD}\in\tau$ then yields a component collection $\cC\subset\CC{X_\tau}$ satisfying $\rsep{Y}{\cD}\le\rsep{X_\tau}{\cC}\in\tau$.
In~particular, $\cO_{\modG}(X_\tau,\cC)\subset\cO_{\modG}(Y,\cD)$ (Lemma~\ref{tcOpenProp} again) avoids all $A_j$ with $j\neq i$.

Letting $O_i:=\bigcup\,\{\,O_{i,\tau}\mid \tau\in A_i\,\}$ for every $i\in I$, we claim that the collection $\{\,O_i\mid i\in I\,\}$ is as desired.
For this, it suffices to show that for all indices $i\neq j$ and ultrafilter tangles $\tau\in A_i$ and $\tau'\in A_j$ the open neighbourhoods $O_{i,\tau}$ and $O_{j,\tau'}$ are disjoint.
By Theorem~\ref{TheTreeSetAPC}~(i) and by symmetry, only the following three cases can possibly occur.

In the first case we have $X_\tau=X_{\tau'}$ and write $X=X_\tau$.
Then $O_{i,X}$ and $O_{j,X}$ are disjoint, ensuring that $\cC(\tau)$ and $\cC(\tau')$ are disjoint.
(If we had not involved the open sets $O_{i,X}$ and $O_{j,X}$, then the component collections $\cC(\tau)$ and $\cC(\tau')$ might possibly have a non-empty finite intersection.)
In particular, $O_{i,\tau}$~and~$O_{j,\tau'}$ are disjoint as well.

In the second case we have $X_\tau\neq X_{\tau'}$ and $\tlsep{X_\tau}\le\trsep{X_{\tau'}}$, which implies that $O_{i,\tau}$ and~$O_{j,\tau'}$ are disjoint.

In the third case we have $X_\tau\neq X_{\tau'}$ and
\begin{align*}
    \trsep{X_\tau}\le\rsep{X_\tau}{C}\le\trsep{X_{\tau'}}
\end{align*}
where $C$ is the component $C_{X_\tau}(X_{\tau'})$.
Since $O_{i,\tau}$ avoids $A_j\ni\tau'$ we deduce that the component $C$ is not contained in $\cC(\tau)$. 
Hence $\lsep{\cC(\tau)}{X_\tau}\le\rsep{X_{\tau'}}{\cC(\tau')}$ which implies that $O_{i,\tau}$ and $O_{j,\tau'}$ are disjoint.
%
\end{proof}

\section{Consistent orientation and lifting from torsos}\label{sec:graphfromtorso}

\noindent For this section, fix a graph $G$, a regular tree set $N$ of finite-order separations of~$G$, and a consistent orientation $O$ of $N$. Also define $\Pi=\bigcap_{(C,D)\in O}D$.

This section deals with the problem of translating separations of $\torso(G,O)$ to separations of $G$, as described in Section~\ref{sec:strategy}.
More precisely, given a separation $(A,B)$ of $\torso(G,O)$, we want to find an \emph{extension} of it in $G$, a separation $(U,W)$ of $G$ towards which all elements of $O$ point such that $U\cap W\subseteq \Pi$ and $(U\cap \Pi,W\cap \Pi)=(A,B)$. 
Note that every extension $(U,W)$ of $(A,B)$ satisfies $U\cap W=A\cap B$.
In general, extensions are not unique.
However, the information contained in $O$ already puts strong restrictions on the structure of extensions.

On the one hand, if $x$ and $y$ are vertices of $G$ and $(C,D)$ is a separation in $O$ with $\{x,y\}\subseteq C$ then every extension $(U,W)$ of a separation of $\torso(G,O)$ has to satisfy $(C,D)\leq (U,W)$ or $(C,D)\leq (W,U)$ and thus $\{x,y\}\subseteq U$ or $\{x,y\}\subseteq W$.
So here we have a relation on $\bigcup_{(C,D)\in O}C$ and related vertices cannot be separated by extensions of separations of $\torso(G,O)$. 

On the other hand, if $(C,D)$ and $(C',D')$ are separations in $O$ such that $O$ also contains some $(C'',D'')$ with $(C,D)\leq (C'',D'')$ and $(C',D')\leq (C'',D'')$, then $(C,D)$ and $(C',D')$ cannot lie on different sides of $(U,W)$ because $(C'',D'')$ points towards every extension $(U,W)$ of $(A,B)$.
So here we have a relation on~$O$ and no extension of a separation of $\torso(G,O)$ can separate two related separations in $O$. 

It turns out that the two relations describe two points of view on the same idea:
In this paper we define~$\sim$ as a relation on the set of separations of $O$, as that fits better in our framework of tree sets.
But it is possible just as well to work with the relation on vertices, as is done e.g.\ in \cite{carmesin2014all}, and several lemmas in this section are inspired by similar lemmas in that paper.
Indeed, we will associate with every equivalence class $\gamma$ of $\sim$ a set of vertices $A_{\gamma}$, thereby associating an equivalence class of $\sim$ of separations with an equivalence class of vertices, and we will work with both $\gamma$ and $A_{\gamma}$.

\begin{lemma}\label{corridorsexist}
Define a relation $\sim$ on $O$ where $(C,D)\sim (C',D')$ if and only if there is a separation in $O$ above both $(C,D)$ and $(C',D')$.
Then $\sim$ is an equivalence relation.
\end{lemma}

\begin{proof}
    By definition the relation is reflexive and symmetric.
    In order to show transitivity, assume that $(C,D)$ and $(C',D')$ are related, as witnessed by $(U,W)\in O$, and that $(C',D')$ and $(C'',D'')$ are related, as witnessed by $(U',W')\in O$.
    As $O$ is a consistent orientation, we have $(U,W)\leq (U',W')$ or $(W',U')\leq (W,U)$ or $(U,W)\leq (W',U')$.
    But $(U,W)\leq (W',U')$ implies $(C',D')\leq (U,W)\leq (W',U')\leq (D',C')$ and thus that $(C',D')\leq (D',C')$ which contradicts the fact that $N$ is regular.
    So either $(U,W)\leq (U',W')$ or $(U',W')\leq (U,W)$ and in both cases the bigger one of these separations shows that $(C,D)$ and $(C'',D'')$ are related.
\end{proof}

\begin{definition}\label{def:corridor}
    An equivalence class of the relation from Lemma~\ref{corridorsexist} is a \emph{corridor} of $O$.
    For a corridor $\gamma$ let $A_{\gamma}$ be the union of all sets $C$ where $(C,D)\in \gamma$.
\end{definition}

\begin{rem}\label{supremumofcorridor}
Let $\gamma$ be a corridor and $(A,B)$ the supremum of all elements of $\gamma$.
Then $A=A_{\gamma}$ and $A\cap B=A_{\gamma}\cap \Pi$.
\end{rem}

\begin{rem}
Lemma~\ref{corridorsexist} also holds in abstract separation systems with the same proof. In particular, corridors are well-defined for abstract separation systems.
\end{rem}

\begin{lemma}\label{corridorcomparable}
    If $(C,D)$ and $(C',D')$ are elements of $O$ and $C\setminus D'$ is non-empty then $(C,D)$ and $(C',D')$ are comparable.
\end{lemma}
\begin{proof}
    Because $O$ is consistent and nested, any two separations in $O$ either point towards each other or are comparable. Let $w$ be a vertex contained in $C\setminus D'$. Then $w$ witnesses that $(C,D)\nleq (D',C')$, hence $(C,D)$ and $(C',D')$ do not point towards each other.
\end{proof}

\begin{lemma}\label{separatefiniteset}
    Let $\gamma$ be a corridor of $O$ and $U$ a finite subset of $A_{\gamma}$. Then there is a separation $(C,D)$ in~$\gamma$ such that $C$ contains $U$ and $C\setminus D$ contains $U\setminus \Pi$.
\end{lemma}
\begin{proof}
    First we consider the special case that $U$ contains only one vertex $v\notin \Pi$. As $v$ is a vertex of $A_{\gamma}$ there is a separation $(C,D)$ in $\gamma$ such that $v\in C$. Furthermore because $v$ is not contained in $\Pi$ there is a separation $(C',D')$ in $O$ such that $v$ is contained in $C'\setminus D'$. By Lemma~\ref{corridorcomparable} the separations $(C,D)$ and $(C',D')$ are comparable and thus contained in the same corridor, so $(C',D')$ is contained in $\gamma$.
    
    Now consider an arbitrary finite subset $U$ of $A_{\gamma}$. For every vertex $v$ of $U$ there is a separation $(C_v,D_v)$ in $\gamma$ such that $v\in C_v$. We just showed that if $v$ is not contained in $\Pi$ then $(C_v,D_v)$ can be chosen such that $D_v$ does not contain $v$. As $\gamma$ is a corridor and $U$ is finite, there is a separation $(C,D)$ in $\gamma$ which is bigger than or equal to all separations $(C_v,D_v)$. In particular $v\in C_v\subseteq C$ for all $v\in U$ and $v\in C_v\setminus D_v\subseteq C\setminus D$ for all $v\in U\setminus \Pi$.
\end{proof}

\begin{lemma}\label{equivalenceofvertices}
    The sets $A_{\gamma}\setminus \Pi$ partition $V(G)\setminus \Pi$.
\end{lemma}
\begin{proof}
    By definition of $\Pi$ every vertex $v\in V(G)\setminus \Pi$ is contained in $A_{\gamma}$ for some corridor $\gamma$, so the sets $A_{\gamma}\setminus \Pi$ cover $V(G)\setminus \Pi$. To prove their disjointness, assume that some vertex is contained in $A_{\gamma}$ and $A_{\gamma'}$ for two corridors $\gamma$ and $\gamma'$. By Lemma~\ref{separatefiniteset} there are separations $(C,D)$ in $\gamma$ and $(C',D')$ in $\gamma'$ respectively such that both $C\setminus D$ and $C'\setminus D'$ contain $v$. Thus by Lemma~\ref{corridorcomparable} the separations $(C,D)$ and $(C',D')$ are contained in the same corridor and hence $\gamma=\gamma'$.
\end{proof}
\begin{corollary}\label{twodefsofcorridorequivalent}
A separation $(C,D)$ of $N$ is contained in a given corridor $\gamma$ if and only if $C\setminus D\subseteq A_{\gamma}$.\qed
\end{corollary}

\begin{lemma}\label{connectedsetrelated}
    Let $U$ be a connected set of vertices avoiding $\Pi$. Then there is a corridor $\gamma$ with $U\subseteq A_{\gamma}$.
\end{lemma}
\begin{proof}
    By Lemma~\ref{equivalenceofvertices} it is sufficient to show the statement for $U$ with exactly two elements. Let $v$ and $w$ be two neighbours not in $\Pi$, and let $(C,D)$ be a separation in $O$ such that $v\in C\setminus D$. Because $w$ is a neighbour of $v$ and $(C,D)$ is a separation, $w$ is contained in $C$ and thus for the corridor $\gamma$ containing $(C,D)$ we have that $A_{\gamma}$ contains both $v$ and $w$.
\end{proof}

\begin{corollary}\label{separatefiniteconnectedset}
    Let $F$ be a finite connected set of vertices not meeting $\Pi$. Then there is a separation $(C,D)\in O$ such that $F\subseteq C\setminus D$.
\end{corollary}
\begin{proof}
    By Lemma~\ref{connectedsetrelated} we may apply Lemma~\ref{separatefiniteset}.
\end{proof}

\begin{lemma}\label{supremumhascliqueseparator}
	Let $\gamma$ be a corridor and assume that all separators of separations in $N$ are cliques. Then $A_{\gamma}\cap \Pi$ is a clique, too.
\end{lemma}
\begin{proof}
	Let $v$ and $w$ be two distinct vertices of $A_{\gamma}\cap \Pi$. Then by Lemma~\ref{separatefiniteset} there is a separation $(C,D)\in \gamma$ such that $C$ contains both $v$ and $w$. Because $v$ and $w$ are contained in $\Pi$ which in turn is a subset of $D$, both $v$ and $w$ are contained in $C\cap D$. Because $C\cap D$ is a clique by assumption, $v$ is a neighbour of $w$.
\end{proof}

\section{Extending the tree set of the principal vertex sets}\label{sec:MainProof}

\noindent In this section we prove our main result, Theorem~\ref{TreeSetForInfTangles}.
To obtain a starting tree set $T$ as described in our overall proof strategy in Section~\ref{sec:strategy}, we apply our technical main result Theorem~\ref{TheTreeSetAPC} (combined with Theorem~\ref{criticalAdmissableExists}) to a carefully chosen collection $\cY$ of \principal\ vertex sets of $G$.
For choosing $\cY$ we need the following definition:

\begin{definition}
A separation $\sep{X}{\cC}$ is \emph{generous} if both $\cC$ and the complement $\cC_X\setminus\cC$ contain components whose neighbourhoods are precisely equal to $X$, i.e.\ if $\CX$ meets both $\cC$ and $\cC_X\setminus\cC$.
A set $X$ of vertices of $G$ is \emph{generous} if it is the separator of some generous separation, i.e.\ if $\vert\CX\vert\ge 2$.
\end{definition}

Now we are ready to set up our starting tree set $T$ and more, as follows.
\textbf{Throughout this section we fix the following notation.}
We let $\cY$ be the collection of all generous subsets of the critical vertex sets of $G$, in formula:
\begin{align*}
    \cY=\big\{\,X\in\cX\;\big\vert\; X\text{ is generous and }\exists\, Y\in\crit(G):X\subset Y\,\big\}.
\end{align*}
Notably, $\crit(G)\subset\cY$.
We assume, without loss of generality by Lemma~\ref{principalCliques}, that each $X\in\cY$ induces a clique $G[X]$.
Using Theorems~\ref{criticalAdmissableExists} and~\ref{TheTreeSetAPC} we obtain a strongly \admissable\ function $\ccK$ for $\cY$ that deviates from all $\CX$ with $X\in\cY$ by precisely one component, in formula $\vert\CX\setminus\ccK(X)\vert=1$ for all $X\in\cY$.
This way we ensure that $T:=T(\cY,\ccK)$ is a regular tree set of generous finite-order separations of $G$.
For $X\in\cY$ we abbreviate $\sigma_X=\sigma_X^{\ccK}$.
Moreover, $O$ always denotes a consistent orientation of $T$, and then $\Pi\subset V(G)$ denotes the part of $O$.
At some point in this section the concept of a `modified torso' of $O$ will be defined. From that point onward, $H$ will always denote the modified torso of $O$.
Whenever we speak of $\Pi$ or $H$ we tacitly assume that they stem from some $O$.
This completes the list of fixed notation for this section.

Next, we consider two inequivalent $\aleph_0$-tangles $\tau_1$ and $\tau_2$ of $G$, we pick a finite-order separation $\{A_1,A_2\}$ of $G$ that efficiently distinguishes $\tau_1$ and $\tau_2$, and we write $Z=A_1\cap A_2$ for its separator.
If $Z$ is included entirely in a critical vertex set of $G$, then $T$ efficiently distinguishes $\tau_1$ and $\tau_2$:

\begin{lemma}\label{RelevantImpliesGenerous}\label{CaseOne}
Let $\sep{Z}{\cD}$ efficiently distinguish two $\aleph_0$-tangles $\tau_1$ and $\tau_2$ of $G$.
Then $\sep{Z}{\cD}$ is generous.
If~additionally $\tau_1$ and $\tau_2$ are inequivalent and $Z$ is included in some critical vertex set of $G$, then $T$ efficiently distinguishes $\tau_1$ and $\tau_2$.
\end{lemma}
\begin{proof}
Let $\{\cD_1,\cD_2\}:=\{\cD,\cC_Y\setminus\cD\}$ such that $\rsep{Z}{\cD_i}\in\tau_i$ for both $i=1,2$.
Our proof starts with a more general analysis of the situation, as follows.
Consider any $i\in \{1,2\}$ and put $j=3-i$. 

If $\tau_i$ lives in a component $C$ of $G-Z$ in that $\rsep{Z}{C}\in\tau_i$, then by the consistency of $\tau_j$ we deduce from $\lsep{C}{Z}\le\rsep{Z}{\cD_j}\in\tau_j$ that $\lsep{C}{Z}\in\tau_j$, so $\sep{Z}{C}$ distinguishes $\tau_1$ and $\tau_2$.
But then so does $\sep{N(C)}{C}$ by Lemma~\ref{FiniteEditDistance}, and hence $N(C)=Z$ follows by the efficiency of $Z$.

Otherwise $\tau_i$ is an ultrafilter tangle and $X_i:=X_{\tau_i}$ is contained in $Z$.
Then, as $U(\tau_i,Z)$ is a free ultrafilter, we have $\rsep{Z}{\cD_i'}\in\tau_i$ for $\cD_i':=\cD_i\cap\cC_Z(X_i)$.
Hence $\rsep{X_i}{\cD_i'}\in\tau_i$ by Lemma~\ref{FiniteEditDistance}.
And $\lsep{\cD_i'}{X_i}\le\rsep{Z}{\cD_j}\in\tau_j$ implies $\lsep{\cD_i'}{X_i}\in\tau_j$ by the consistency of $\tau_j$.
Therefore, $\sep{X_i}{\cD_i'}$ distinguishes $\tau_1$ and $\tau_2$, so $X_i=Z$ follows by the efficiency of $Z$.

From the two cases above we deduce that $\sep{Z}{\cD}$ is generous.
It remains to show that if additionally $\tau_1$ and $\tau_2$ are inequivalent and $Z$ is contained in a critical vertex set of $G$, then $T$ efficiently distinguishes $\tau_1$ and~$\tau_2$.
First, we have $Z\in\cY$ as $Z$ is generous.
Next, we note that not both $\tau_1$ and $\tau_2$ can be ultrafilter tangles with $X_1,X_2\subset Z$ for otherwise $X_1=Z=X_2$ follows from our considerations above, contradicting that $\tau_1$ and $\tau_2$ are inequivalent.
So at least one of $\tau_1$ and $\tau_2$ lives in a component $C$ of $G-Z$, say $\rsep{Z}{C}\in\tau_1$, and then $C\in\CZ$ follows from our considerations above.
If $C\in\ccK(Z)$ then $\sep{Z}{C}\in T$ efficiently distinguishes $\tau_1$ and $\tau_2$.
Otherwise $\{C\}=\CZ\setminus\ccK(Z)$, and we claim that $\tsep{Z}\in T$ efficiently distinguishes $\tau_1$ and $\tau_2$.
On the one hand, $\tlsep{Z}\le\rsep{Z}{C}\in\tau_1$ implies $\tlsep{Z}\in\tau_1$ by the consistency of $\tau_1$.
On the other hand, $\tau_2$ either lives in a component in $\ccK(Z)$ or $\tau_2$ is an ultrafilter tangle with $X_2=Z$, and in both cases we deduce $\trsep{Z}\in\tau_2$.
\end{proof}

Therefore, we may assume that $Z$ is not contained entirely in any critical vertex set of $G$.
Then $Z$ is contained in a part of $T$, as follows.

\begin{lemma}\label{GenerousIsPrincipalForPrincipals}
Let $Z\in\cX$ be generous.
If $X$ is a \principal\ vertex set of $G$ that does not contain $Z$ entirely, then there is a unique component of $G-X$ that $Z$ meets.
\end{lemma}
\begin{proof}
As $Z$ is not contained in $X$ as a subset, there is a component $C$ of $G-X$ which $Z$ meets.
Assume for a contradiction that there is another component $D$ of $G-X$ meeting $Z$.
Pick vertices $c\in Z\cap C$ and $d\in Z\cap D$.
Now note that every component $K\in\CZ$ must meet $X$, for $K$ plus its $K$--$c$ and $K$--$d$ edges admits a $c$--$d$ path connecting the distinct components $C$ and $D$ of $G-X$.
But since $X$ is \principal\ it meets at most one component of $G-Z$, namely $C_Z(X)$, contradicting that $\vert\CZ\vert\ge 2$.
\end{proof}

By Lemma~\ref{GenerousIsPrincipalForPrincipals} above the separator $Z$ meets precisely one side from every separation in $T$, and then orienting each separation in $T$ towards that side results in a consistent orientation $O$ of $T$ whose part $\Pi$ contains $Z$.
The remainder of this section is dedicated to modifying the torso $H$ of $O$ so that 
\begin{itemize}
    \item $\tau_1$ and $\tau_2$ are `represented' by `proxy' ends $\eta_1$ and $\eta_2$ in $H$; and
    \item applying Carmesin's theorem in $H$ yields a tree set that lifts compatibly with $T$ to a tree set of tame finite-order separations of $G$ that efficiently distinguishes $\tau_1$ and $\tau_2$.
\end{itemize}

\subsection{Modified torsos, proxies of corridors and lifting from modified torsos}

In this subsection we introduce modified torsos and show that there is an elegant way to lift tree sets from modified torsos to the graph $G$ itself.
Proxies of corridors are introduced as a technical tool whose purpose is twofold: first, they are key to the elegant lifting of tree sets.
And second, they will be employed in the next subsection to define proxies for $\aleph_0$-tangles.

\begin{definition}[Modified torso]
Whenever $\Pi$ is non-empty we define the \emph{modified torso} $H$ of $O$, as follows.
Consider the set $\cZ$ of all finite subsets of $\Pi$ that are separators of suprema of corridors of $O$.
Then we obtain $H$ from $G[\Pi]$ by disjointly adding for each $X\in\cZ$ a copy of $K^{\aleph_0}$ that we join completely to $X$.
\end{definition}

We remark that $\Pi$ being non-empty ensures that the empty set is not an element of $\cZ$, so modified torsos are connected.
Since the copies of $K^{\aleph_0}$ are joined to finite cliques of $G[\Pi]$ by Lemma~\ref{supremumhascliqueseparator}, no two ends of $G[\Pi]$ are merged when we move on to the modified torso $H$:

\begin{lemma}\label{endSpaceModifiedTorso}
Every finite-order separation of $G[\Pi]$ extends to a finite-order separation of $H$.
Thus sending each end $\eta$ of $G[\Pi]$ to the end $\iota(\eta)$ of $H\supset G[\Pi]$ with $\eta\subset\iota(\eta)$ defines an injection $\iota\colon\Omega(G[\Pi])\hookrightarrow\Omega(H)$.
Moreover, the ends of $H$ that do not lie in the image of $\iota$ correspond bijectively to the copies of $K^{\aleph_0}$ that were added to $G[\Pi]$ in order to obtain $H$.\qed
\end{lemma}

Now we tend to the lifting of separations from $H$ to $G$.
It is desirable to have the separator of a separation remain unchanged when lifting it.
But $H$ usually will contain many vertices that are not vertices of the original graph $G$.
We solve this as follows.
When we consider finite-order separations of $H$, we are only interested in ones that efficiently distinguish some two ends of $H$.
And these $H$-relevant separations have their separators consist of vertices of the original graph $G$:

\begin{definition}[$H$-relevant]
If a separation of $H$ has finite order and efficiently distinguishes some two ends of $H$, then we call it and its orientations $H$-\emph{relevant}.
\end{definition}

\begin{lemma}\label{HrelevantIsReasonable}
If $\{A,B\}$ is $H$-relevant, then $A\cap B\subset\Pi$.
\end{lemma}
\begin{proof}
Assume for a contradiction that $A\cap B$ meets an added copy $K$ of a $K^{\aleph_0}$ in a vertex $v$ and write $X=N_H(K)$.
Notably $H[X]$ is a clique, and hence so is $H[X\cup K]$.
Without loss of generality we may assume that $H[X\cup K]\subset H[A]$, so $K$ meets $A\setminus B$ while $H[X\cup K]$ avoids $B\setminus A$.
Now $v\in A\cap B\cap K$ sends its edges only to $K$ and~$X$, and in particular $v$ sends no edges to $B\setminus A$.
So $\{A,B-v\}$ is again a separation of $H$, but of order $\vert A\cap B\vert-1$, and this separation still distinguishes all the ends of $H$ that were distinguished by $\{A,B\}$, contradicting that $\{A,B\}$ is $H$-relevant.
\end{proof}

Now we are almost ready to define lifts of separations of $H$, all we are missing is 

\begin{definition}[Proxy of a corridor]
Suppose that $\Pi$ is non-empty and $\gamma$ is a corridor of $O$.
The \emph{proxy} of $\gamma$ in the modified torso $H$ is the end $\eta$ of $H$ that is defined as follows.
Consider the separator $X$ of the supremum of $\gamma$.
If $X$ is finite, then $\eta$ is the end of $H$ containing the rays of the $K^{\aleph_0}$ that was added for~$X$.
Otherwise $G[X]\subset H$ is an infinite clique by Lemma~\ref{supremumhascliqueseparator}, and then $\eta$ is the end of $H$ that contains the rays of $G[X]$.
\end{definition}

Finally, we can lift separations from $H$ to $G$:

\begin{definition}[Lift from a modified torso]
Let $(A,B)$ be an $H$-relevant separation of a modified torso~$H$.
By Lemma~\ref{HrelevantIsReasonable} the separator $A\cap B$ is included in $\Pi$ entirely.
The \emph{lift} $(\ell(A),\ell(B))$ of $(A,B)$ is defined as follows.
The set $\ell(A)\subset V(G)$ agrees with $A$ on $\Pi$, and a vertex of $G-\Pi$ is contained in $\ell(A)$ whenever its corridor's proxy in $H$ lives on the $A$-side.
The set $\ell(B)$ is defined analogously.
\end{definition}

We remark that $\{\ell(A),\ell(B)\}$  does not depend on the orientation of $\{A,B\}$.
In order to verify that the lifts work as intended we need the following lemma:

\begin{lemma}\label{liftRespectsAgamma}
If $\{A,B\}$ is $H$-relevant and $\gamma$ is a corridor of $O$ whose proxy lives on the $A$-side, then $A_\gamma\subset\ell(A)$.
\end{lemma}
\begin{proof}
We have $A_\gamma\setminus\Pi\subset\ell(A)$ by the definition of $\ell(A)$.
It remains to show $\Xi\subset A$ for the separator $\Xi=A_\gamma\cap\Pi$ of the supremum of $\gamma$.
If $\Xi$ is infinite, then the proxy $\eta$ of $\gamma$ living on the $A$-side means $G[\Xi]\subset H[A]$.
Otherwise $\Xi$ is finite.
Then $\eta$ stems from a copy $K\subset H$ of $K^{\aleph_0}$ that is joined completely to the clique $G[\Xi]$, and so $\eta$ living on the $A$-side means $K\subset H[A]$.
Consequently, the infinite clique $H[K\cup \Xi]$ is contained in $H[A]$ as well, giving $\Xi\subset A$ as desired.
\end{proof}

Now we can check for ourselves that lifts work:

\begin{lemma}\label{liftisseparation}
The lift of an $H$-relevant separation is a separation of $G$ with the same separator.
\end{lemma}
\begin{proof}
Let $\{A,B\}$ be any $H$-relevant separation, and recall that $A\cap B\subset\Pi$ by Lemma~\ref{HrelevantIsReasonable}.
Every vertex of $G-\Pi$ lies in $A_\gamma$ for a unique corridor $\gamma$ of $O$, and hence is contained in precisely one of $\ell(A)$ and $\ell(B)$.
Thus $\ell(A)\cap\ell(B)=A\cap B$.
It remains to verify that $G$ has no edge between $\ell(A)\setminus \ell(B)$ and $\ell(B)\setminus\ell(A)$.
For this, let $e=xy$ be any edge of $G$.
If both $x$ and $y$ are contained in $\Pi$, then $e\subset A$ say, and hence $e\subset\ell(A)$.
Otherwise one of $x$ and $y$ lies outside of $\Pi$, say $x\in\ell(A)\setminus\Pi$.
Let $\gamma$ be the corridor of $O$ with $x\in A_\gamma\setminus\Pi$, so the proxy $\eta$ of $\gamma$ lives on the $A$-side.
From $x\in A_\gamma\setminus\Pi$ we infer $y\in A_\gamma$.
Then $e\subset A_\gamma\subset\ell(A)$ by Lemma~\ref{liftRespectsAgamma}.
\end{proof}

Starting with an intuitive lemma we verify that our lifts are compatible with $T$ and lifts of other modified torsos:

\begin{lemma}\label{liftCommutesWithProxyOfCorridors}
Let $\gamma$ be a corridor of $O$ and let $\eta$ be the proxy of $\gamma$ in $H$.
If $\{A,B\}$ is $H$-relevant with $\eta$ living on the $A$-side, then $\vs\leq (\ell(A),\ell(B))$ for all $\vs\in\gamma$.
In particular, the lift of an $H$-relevant separation is nested with $T$.
\end{lemma}
\begin{proof}
Consider any $(C,D)\in\gamma$.
We have to show $(C,D)\le (\ell(A),\ell(B))$.
For the inclusion $C\subset\ell(A)$ we start with $C\subset A_\gamma$ and employ Lemma~\ref{liftRespectsAgamma} for $A_{\gamma}\subset\ell(A)$.
Now the inclusion $\ell(B)\subset D$ is tantamount to $C\setminus D\subset \ell(A)\setminus\ell(B)$ which is immediate from $C\subset\ell(A)$ as $C\setminus D$ avoids $\Pi\supset \ell(A)\cap\ell(B)$ (cf.~Lemma~\ref{liftisseparation}).
\end{proof}

\begin{corollary}\label{liftsOfDistinctTorsosAreNested}
If $H'$ is the modified torso of a consistent orientation $O'$ of $T$ other than $O$, then all lifts of $H$-relevant separations are nested with all lifts of $H'$-relevant separations.\qed
\end{corollary}

\begin{lemma}\label{liftRespectsNested}
If $(A,B)$ and $(C,D)$ are two $H$-relevant with $(A,B)\le (C,D)$, then their lifts satisfy $(\ell(A),\ell(B))\le (\ell(C),\ell(D))$.
In particular, the lifts of two nested $H$-relevant separations are again nested.\qed
\end{lemma}

We close this subsection with the lemma that ensures that when we construct the tree set for our main result, we are able to ensure the `moreover' part stating equivalent $\aleph_0$-tangles orient the tree set the same~way.

\begin{lemma}\label{LiftsAreTame}
Every $H$-relevant separation lifts to a tame separation of $G$.
\end{lemma}
\begin{proof}
If $H$ stems from a consistent orientation $O$ of $T$ that contains some star $\sigma_X$ with $X\in\cY$, then $O=\dc{\sigma_X}$ (with the down-closure taken in $\vT$) by Theorem~\ref{TheTreeSetAPC}. Consequently, $H$ was obtained from the finite clique $G[X]$ by disjointly adding precisely one copy of a $K^{\aleph_0}$ and joining it completely to $X$, so the one-ended $H$ has no $H$-relevant separations.
Therefore, we may assume that $O$ avoids all of the stars $\sigma_X$ with $X\in\cY$.

Let $\{A,B\}$ be an $H$-relevant separation and recall that $A\cap B\subset\Pi$ by Lemma~\ref{HrelevantIsReasonable}.
And let $X\subset A\cap B$ be a critical vertex set of $G$.

If there is a component $K\in\ccK(X)$ with $\rsep{X}{K}\in O$, then the proxy of the corridor of $O$ that contains $\rsep{X}{K}$ ensures that all the components in $\ccK(X)\setminus\{K\}$ are contained in the same side of $\{\ell(A),\ell(B)\}$.

Otherwise, since $O$ avoids the star $\sigma_X$, we have $\tlsep{X}\in O$.
Then the proxy of the corridor of $O$ that contains $\tlsep{X}$ ensures that all the components in $\ccK(X)$ are contained in the same side of $\{\ell(A),\ell(B)\}$.

In either case, all but finitely many of the components in $\CX$ lie on the same side of $\{\ell(A),\ell(B)\}$.
Since $A\cap B\supset X$ meets at most finitely many components in $\CX$, the collection $\cC_{A\cap B}(X)$ forms a cofinite subset of $\CX$, and therefore all but finitely many components in $\cC_{A\cap B}(X)$ lie on the same side of $\{\ell(A),\ell(B)\}$ as desired.
\end{proof}

\subsection{\texorpdfstring{Proxies of $\aleph_0$-tangles}{Proxies of infinite tangles}}

We start this subsection by introducing the technical notion of `walking a corridor' and prove two technical lemmas about ends.
This framework, together with proxies of corridors, then enables us to give a comprehensible definition of proxies of $\aleph_0$-tangles.
We emphasise that this technical layering is highly important to save the key segments of our overall proof from being swamped with terrible amounts of case distinctions.

\begin{definition}[Walking]
We say that an end $\omega$ of $G$ \emph{walks} a corridor $\gamma$ of $O$ if for the supremum $(A,B)$ of $\gamma$ the end $\omega$ has a ray contained in $G[A\setminus B]$.
And we say that an ultrafilter tangle $\tau$ of $G$ \emph{walks} a corridor~$\gamma$ of $O$ if $\tau$ contains the inverse of some separation in $\gamma$.
\end{definition}

\begin{lemma}\label{raytowardsinfinitesupremum}
Suppose that $N$ is a tree set of generous finite-order separations of $G$ all whose separators induce cliques.
Let $\omega$ be an end of $G$, let $\Pi$ be the part of the orientation $O=\omega\cap\vN$ that $\omega$ induces on $N$, and suppose that $\omega$ walks a corridor $\gamma$ of $O$.
If the separator $A_\gamma\cap\Pi$ of the supremum of $\gamma$ is infinite, then $G[\Pi]$ contains a ray from $\omega$.
\end{lemma}
\begin{proof}
Pick $R\in\omega$ arbitrarily.
	By Lemma~\ref{supremumhascliqueseparator} it is sufficient to show that there are infinitely many pairwise disjoint paths from $R$ to $A_{\gamma}\cap \Pi$. 
	We will recursively construct such paths $P_n$ ($n\in\N$) of which only the last vertex $v_n$ is contained in $\Pi$. 
	Assume that $P_0,\ldots ,P_{n-1}$ have already been defined. 
	Then there is a finite non-empty initial segment $R'$ of $R$ such that $R'\cup P_0\mathring{v}_0\cup \cdots \cup P_{n-1}\mathring{v}_{n-1}$ is connected. Let $(A,B)\in O$ be a separation such that all vertices of $R'\cup P_0\mathring{v}_0\cup \cdots \cup P_{n-1}\mathring{v}_{n-1}$ are contained in $A\setminus B$ (such a separation exists by Corollary~\ref{separatefiniteconnectedset}). 
	Then $(A,B)$ is contained in $\gamma$. 
	Every vertex $v_k$ with $k<n$ is a neighbour of a vertex in $A\setminus B$ and thus contained in $A$.
	
	As $A_{\gamma}\cap \Pi$ is infinite, it contains a vertex $v$ which is not contained in $A\cap B$ and thus not contained in~$A$. 
	In particular, $v$ is not contained in any path $P_k$ with $k<n$. 
	Because $v\in A_{\gamma}$, there is a separation $(A',B')$ in $\gamma$ such that $v\in A'$ and thus $v\in A'\cap B'$. 
	Let $(C,D)$ be a separation in $\gamma$ which is bigger than both $(A,B)$ and $(A',B')$. 
	Then all vertices contained in some $P_k$ with $k<n$ are contained in $(C\setminus D)\cup \Pi$. Furthermore $v\in C\cap D$ and $R$ contains a vertex of $C\setminus D$. 
	Because $(C,D)\in O$, some tail of $R$ is contained in $D\setminus C$ and thus $R$ also contains vertices of $D\setminus C$. 
	As $(C,D)$ is a separation and $R$ connected this implies that some vertex $w$ of $R$ is contained in $C\cap D$. 
	Because $w$ is not a vertex of $\Pi$ it is also not a vertex of some $P_k$ with $k<n$.
	
	As $(D,C)$ is generous, there is a component of $G-(C\cap D)$ which is contained in $D\setminus C$ and whose neighbourhood is precisely equal to $C\cap D$. 
	Thus there is a path $P$ from $w$ to $v$ whose inner vertices are contained in $D\setminus C$. 
	We already established that $v$ and $w$ are not vertices of any $P_k$ with $k<n$. 
	Hence $P$ is disjoint from all $P_k$ with $k<n$. 
	Let $v_n$ be the first vertex of $P$ in $\Pi$ and let $P_n:=w P v_n$. 
	By Corollary~\ref{separatefiniteconnectedset} there is a separation $(I,J)\in O$ such that the vertices of $P_n\mathring{v}_n$ are contained in $I\setminus J$. 
	Then $(I,J)\in \gamma$ and $v_n\in I$, so $v_n\in A_{\gamma}$. 
	As also $v_n\in \Pi$ we have $v_n\in A_{\gamma}\cap \Pi$ as required.
\end{proof}

\begin{lemma}\label{endOutsideClosureGivesCorridor}
If an end $\omega$ of $G$ does not lie in the closure of $\Pi$, then $\omega$ walks a unique corridor of $O$.
\end{lemma}
\begin{proof}
Since $\omega$ does not lie in the closure of $\Pi$, we in particular find a ray $R\in\omega$ that avoids $\Pi$.
As $R$ is connected, it defines a corridor $\gamma$ of $O$ with $R\subset A_\gamma\setminus\Pi$.
Then $\omega$ walks the corridor $\gamma$, and so it remains to show that $\gamma$ is unique.

If $O\not\subset\omega$, then $\omega$ contains the inverse $\sv$ of some separation $\vs\in O$, and then $\gamma$ is determined as the corridor of $O$ containing $\vs$.
Otherwise $O\subset\omega$.
Then we assume for a contradiction that there is another ray $R'\in\omega$ that walks a corridor $\gamma'$ of $O$ other than $\gamma$.
Since the suprema of $\gamma$ and $\gamma'$ both separate $R$ and~$R'$, their separators cannot be finite, and so they are infinite.
But then applying Lemma~\ref{raytowardsinfinitesupremum} to either $\gamma$ or $\gamma'$ yields a ray of $\omega$ in $G[\Pi]$, contradicting the assumption that $\omega$ does not lie in the closure of $\Pi$.
\end{proof}

Finally, we are ready for the definition of proxies of $\aleph_0$-tangles.
We split the definition and consider ends and ultrafilter tangles separately.

\begin{definition}[Proxy of an end]
If $\omega$ is an end of $G$, then the \emph{proxy} of $\omega$ in $H$ is the end $\eta$ of $H$ that is defined as follows.
\begin{itemize}
    \item If $\omega$ lies in the closure of $\Pi$, then $\omega$ has a ray in $G[\Pi]$ by Corollary~\ref{endInClosureOfPartGivesRayInTorso}, and $\eta$ is the end of such a ray in $H$ (this is well-defined by Corollary~\ref{equivalentRaysOfPartToTorso}).
    \item Otherwise $\omega$ does not lie in the closure of $\Pi$ and by Lemma~\ref{endOutsideClosureGivesCorridor} walks a unique corridor $\gamma$ of $O$; then $\eta$ is the proxy of $\gamma$ in $H$.
\end{itemize}
\end{definition}

\begin{definition}[Proxy of an ultrafilter tangle]
If $\tau$ is an ultrafilter tangle of $G$ and $O$ avoids the star~$\sigma_{X_\tau}$, then $\tau$ walks a unique corridor $\gamma$ of $O$ and the \emph{proxy} of $\tau$ in $H$ is the end $\eta$ of $H$ that is the proxy of $\gamma$ in~$H$.
\end{definition}

We close this subsection with a lemma on the interaction of lifts with proxies:

\begin{lemma}\label{liftCommutesWithProxyOfTangles}
Let $\tau$ be an $\aleph_0$-tangle of $G$ and suppose that the proxy $\eta$ of $\tau$ in $H$ is defined.
If $\{A,B\}$ is $H$-relevant and $(A,B)\in\eta$, then $(\ell(A),\ell(B))\in\tau$.
\end{lemma}
\begin{proof}
If $\tau$ is an end of $G$ that lies in the closure of $\Pi$, then this follows from the fact that some ray of $G[\Pi]$ is contained in both $\tau$ and $\eta$.
Otherwise $\tau$ is an $\aleph_0$-tangle of $G$ that walks a unique corridor $\gamma$ of $O$.
If~additionally $\tau$ is an end, then every ray in $\tau$ that avoids $\Pi$ is contained in $\ell(B)$, ensuring $(\ell(A),\ell(B))\in\tau$.
So we may assume that $\tau$ is an ultrafilter tangle. 
As the proxy $\eta$ of $\tau$ is defined, we know that $O$ avoids the star $\sigma_{X_\tau}$ so that $\tau$ walks a unique corridor $\gamma$ of $O$.
By definition, this means that $\tau$ contains the inverse of some oriented separation from $\gamma$. 
Then $(\ell(A),\ell(B))\in\tau$ follows from Lemma~\ref{liftCommutesWithProxyOfCorridors} and the consistency of the tangle~$\tau$.
\end{proof}

\subsection{Efficiently distinguishing the proxies}

In this subsection we provide the final key segments of our overall proof.
We start with an overview of the situation that is of interest.

\begin{figure}
    \centering
    \begin{tikzpicture}[decoration=snake,thick,dot/.style={colourdot,circle,fill,inner sep=0.04cm,minimum size=0.1cm},labelgiven/.style={label distance=0cm,inner sep=0.08cm}]
\pgfdecorationsegmentamplitude=1.5pt
\pgfdecorationsegmentlength=0.2cm
\newcommand{\Zwidth}{0.4cm}
\newcommand{\Hwidth}{3cm}
\newcommand{\graphheight}{1cm}
\newcommand{\outerxshift}{2.5cm}
\newcommand{\outeryshift}{0.3cm}
\newcommand{\bendindex}{.3*\outerxshift}
\newcommand{\outeryradius}{2.5cm}
\newcommand{\labelheight}{0.2cm}
\newcommand{\Zangle}{40}
\newcommand{\Gangle}{70}
\draw 
	(0,-\graphheight) .. controls +(\bendindex,0) and +(-\bendindex,0).. ++(\outerxshift,-\outeryshift) arc [start angle=-90, end angle=90,x radius=\outeryradius+\outeryshift, y radius=\graphheight+\outeryshift] node [pos=0.7,above] {$G$} .. controls +(-\bendindex,0) and +(\bendindex,0).. ++(-\outerxshift,-\outeryshift) .. controls +(-\bendindex,0) and +(\bendindex,0) .. ++(-\outerxshift,\outeryshift) arc [start angle=90, end angle=270,x radius=\outeryradius + \outeryshift,y radius=\graphheight+\outeryshift] .. controls +(\bendindex,0) and +(-\bendindex,0) .. ++(\outerxshift,\outeryshift);
\draw[pictureorange] 
	(0,0) ellipse [x radius=\Hwidth,y radius=\graphheight];
\begin{onlayer}{filllayer}
\fill [pictureorange!10]
	(0,0) ellipse [x radius=\Hwidth,y radius=\graphheight];
\end{onlayer}
\draw[red] 
(0,0) ellipse [x radius=\Zwidth,y radius=\graphheight];
\begin{onlayer}{filllayer}
\fill [red!10]
(0,0) ellipse [x radius=\Zwidth,y radius=\graphheight];
\end{onlayer}
\draw [coloura] 
	(0,\graphheight) .. controls +(-\bendindex,0) and +(\bendindex,0) .. ++(-\outerxshift,\outeryshift) arc [start angle=90, end angle=90+\Gangle,x radius=\outeryradius+\outeryshift, y radius=\graphheight+\outeryshift]  ..controls +(240:0.7) and +(270:0.5) .. (canvas polar cs:angle=\Zangle,x radius=\Zwidth, y radius=\graphheight) node [pos=0.32,above, labelgiven] {$A_{\gamma_1}$} node [colourxi,pos=0.56,labelgiven,above] {$\Xi_1$} arc [start angle= \Zangle, end angle=90, x radius=\Zwidth, y radius=\graphheight];
\begin{onlayer}{filllayer}
\fill [coloura!10]
	(-\outerxshift,\graphheight+\outeryshift) arc [start angle=90, end angle=90+\Gangle,x radius=\outeryradius+\outeryshift, y radius=\graphheight+\outeryshift]  ..controls +(240:0.7) and +(270:0.5) .. (canvas polar cs:angle=\Zangle,x radius=\Zwidth, y radius=\graphheight) arc [start angle= \Zangle, end angle=90, x radius=\Zwidth, y radius=\graphheight];
\end{onlayer}
\path
	(canvas polar cs:angle=\Zangle,x radius=\Zwidth, y radius=\graphheight) ++(-0.02cm,0.02cm)
	coordinate (11);
\begin{scope} [colourxi]
	\path [clip]
		(0.5cm, \graphheight+0.05cm) rectangle ++(-3.47cm,-0.93cm);
	\draw [dashed]
		(-\outerxshift,\graphheight+\outeryshift) arc [start angle=90, end angle=90+\Gangle,x radius=\outeryradius+\outeryshift, y radius=\graphheight+\outeryshift] ++(0cm,0.02cm) ..controls +(240:0.7) and +(270:0.5) .. (11);
	\draw [dashed]
		(canvas polar cs:angle=\Zangle,x radius=\Zwidth-0.02cm,y radius=\graphheight - 0.025 cm) arc [start angle=\Zangle, end angle=90,x radius=\Zwidth-0.02 cm,y radius=\graphheight - 0.025cm] arc [start angle=90,end angle=270,x radius=\Hwidth-0.02cm,y radius=\graphheight-0.02cm];
\end{scope}
\draw [coloura] 
(0,-\graphheight) .. controls +(\bendindex,0) and +(-\bendindex,0).. ++(\outerxshift,-\outeryshift) arc [start angle=-90, end angle=-90+\Gangle,x radius=\outeryradius+\outeryshift, y radius=\graphheight+\outeryshift]  ..controls +(60:0.7) and +(90:0.5) .. (canvas polar cs:angle=\Zangle+180,x radius=\Zwidth, y radius=\graphheight) node [pos=0.32,below, labelgiven] {$A_{\gamma_2}$} node [colourxi,pos=0.56,label distance=0cm,inner sep=0.1cm,below] {$\Xi_2$} arc [start angle= \Zangle+180, end angle=270, x radius=\Zwidth, y radius=\graphheight];
\begin{onlayer}{filllayer}
\fill [coloura!10]
(\outerxshift,-\graphheight-\outeryshift) arc [start angle=-90, end angle=-90+\Gangle,x radius=\outeryradius+\outeryshift, y radius=\graphheight+\outeryshift]  ..controls +(60:0.7) and +(90:0.5) .. (canvas polar cs:angle=\Zangle+180,x radius=\Zwidth, y radius=\graphheight)  arc [start angle= \Zangle+180, end angle=270, x radius=\Zwidth, y radius=\graphheight];
\end{onlayer}
\path
(canvas polar cs:angle=\Zangle+180,x radius=\Zwidth, y radius=\graphheight) ++(0.02cm,-0.02cm)
coordinate (22);
\begin{scope} [colourxi]
\path [clip]
(-0.5cm,-\graphheight-0.05cm) rectangle ++(3.47cm,0.93cm) coordinate (corner two);
\draw [dashed]
(\outerxshift,-\graphheight-\outeryshift) arc [start angle=-90, end angle=-90+\Gangle,x radius=\outeryradius+\outeryshift, y radius=\graphheight+\outeryshift] ++(0cm,-0.02cm) ..controls +(60:0.7) and +(90:0.5) .. (22);
\draw [dashed]
(canvas polar cs:angle=\Zangle+180,x radius=\Zwidth-0.02cm,y radius=\graphheight - 0.025 cm) arc [start angle=\Zangle+180, end angle=270,x radius=\Zwidth-0.02 cm,y radius=\graphheight - 0.025cm] arc [start angle=-90,end angle=90,x radius=\Hwidth-0.02cm,y radius=\graphheight-0.02cm];
\end{scope}
\draw[colourxi,dashed] 
	(canvas polar cs:angle=180+\Zangle, x radius=\Zwidth, y radius=\graphheight) ++(0.02cm,0.03cm)-- (-0.2cm,-\graphheight-2*\outeryshift-0.2cm);
\draw [colourxi,dashed]
	(corner two) -- (0.6*\Hwidth + 0.2cm,-\graphheight-2*\outeryshift-0.2cm) node [anchor=west,labelgiven] {$K^{\aleph_0}\subseteq H\setminus G$};
\draw[colourxi] 
	(0.3*\Hwidth,-\graphheight-2*\outeryshift-0.2cm) ellipse [x radius=0.3*\Hwidth+0.2cm, y radius=0.2cm];
\draw[colourray,decorate,->]
	(0,0.8*\graphheight) ++(0cm,0cm) -- (-0.4*\Hwidth,0.6*\graphheight) node [labelgiven,anchor=15] {$\eta_1$};
\draw[colourray,decorate,->]
	(0.2cm,-\graphheight-2*\outeryshift-0.2cm) -- ++(0.6*\Hwidth-0.4cm,0);
\path
	(\Zwidth,0cm) node [red,anchor=east,labelgiven] {$Z$}
	(canvas polar cs:angle=45,x radius=\Hwidth, y radius=\graphheight) node [pictureorange,labelgiven,anchor=north east] {$H$}
	(0.2cm,-\graphheight-0.25cm) node [colourray,name=eta two,labelgiven] {$\eta_2$}
	(3*\Zwidth,-0.5*\graphheight) node [dot,name=point in set] {}
	(-1.5*\Zwidth,0.7*\graphheight) ++ (0.06cm,0.06cm) node [name=point on ray,dot] {}
	(-\Zwidth,\graphheight+\outeryshift+0.1cm) node [colourdot,name=exists one,labelgiven] {$\exists$}
	(2*\Zwidth,-\graphheight-0.35cm) node [colourdot,name=exists two,labelgiven] {$\exists$};
\draw [colourdot,thin]
	(point on ray) -- (exists one);
\draw [colourdot,thin]
	(point in set) -- (exists two);
\draw [colourray,thin]
	(0.5cm,-\graphheight-2*\outeryshift-0.2cm) ++(-0.12cm,0.07cm) -- (eta two);
\end{tikzpicture}
    \caption{A graph $G$, a vertex set $Z$ efficiently separating two inequivalent $\aleph_0$-tangles $\tau_1$ and $\tau_2$ of $G$, and a modified torso $H(Z)$ which contains $Z$. The $\aleph_0$-tangles $\tau_i$ walk corridors $\gamma_i$ and $H(Z)$ has proxies $\eta_1$ and $\eta_2$ for $\tau_1$ and $\tau_2$.}
    \label{fig:detailedReflection}
\end{figure}
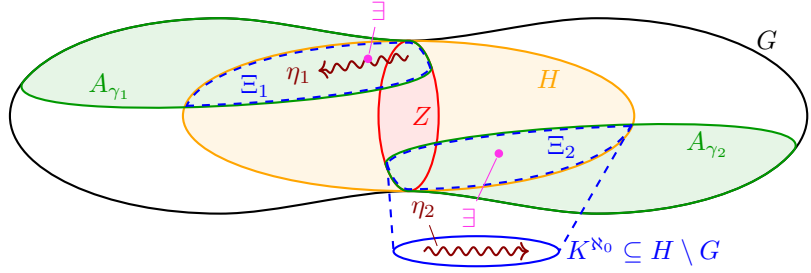

\textbf{Throughout this subsection we fix the following notation in addition to the notation fixed throughout the ambient section.}
(See also Figure~\ref{fig:detailedReflection}.)
We are given two inequivalent $\aleph_0$-tangles $\tau_1$ and $\tau_2$ of $G$ that are efficiently distinguished by a finite-order separation $\{A_1,A_2\}$ of $G$ with separator $Z=A_1\cap A_2$.
The separator $Z$ is not contained in a critical vertex set of $G$.
Hence, by Lemma~\ref{GenerousIsPrincipalForPrincipals} the separator $Z$ meets precisely one side from every separation in $T$, and then orienting each separation in~$T$ towards that side results in a consistent orientation $O$ of $T$ whose part $\Pi$ contains $Z$.
For this special orientation we write $O(Z)$, and we write $\Pi(Z)$ and $H(Z)$ for its part and modified torso.
Moreover, $\eta_1$~and~$\eta_2$ are the proxies of $\tau_1$ and $\tau_2$ in $H(Z)$ (note that these are defined as $O(Z)$ avoids all stars $\sigma_X$ with $X\in\cY$).
Whenever we write $i$ we mean an arbitrary $i\in \{1,2\}$, and we write $j=3-i$.
If $\tau_i$ happens to be an ultrafilter tangle, then we write $X_i$ instead of $X_{\tau_i}$.
This completes the list of fixed notation for this subsection.

The final key segments are Lemma~\ref{proxyLifting} and Proposition~\ref{distinctCorridorsDistinctProxies} below. We start with the lemma:

\begin{lemma}\label{proxyLifting}
Every relevant finite-order separation of $H$ that distinguishes $\eta_1$ and $\eta_2$ does lift to a separation of $G$ that distinguishes $\tau_1$ and $\tau_2$.
\end{lemma}
\begin{proof}
Let $\{A,B\}$ be any relevant finite-order separation of $H$ that distinguishes $\eta_1$ and $\eta_2$, say with $(B,A)\in\eta_1$ and $(A,B)\in\eta_2$.
Then Lemma~\ref{liftCommutesWithProxyOfTangles} gives both $(\ell(B),\ell(A))\in\tau_1$ and $(\ell(A),\ell(B))\in\tau_2$, so $\{\ell(A),\ell(B)\}$ distinguishes $\tau_1$ and $\tau_2$.
\end{proof}

For the key proposition, we need the following proposition whose proof we postpone to after the proof of the key proposition.

\begin{proposition}\label{MagicLemma}
If $\tau_i$ walks a corridor $\gamma_i$ of $O(Z)$ where $\Xi_i$ denotes the separator of the supremum of~$\gamma_i$, then $G[\Xi_i\setminus Z]$ is a non-empty clique that is entirely contained in $G[A_i\setminus A_j]$.
\end{proposition}

The final key segment is

\begin{proposition}\label{distinctCorridorsDistinctProxies}
The proxies $\eta_1$ and $\eta_2$ are efficiently distinguished by $Z$.
\end{proposition}

\begin{proof}
If $Z$ distinguishes the proxies $\eta_1$ and $\eta_2$ in $H(Z)$, then it does so efficiently, for otherwise the separation of order $<\vert Z\vert$ doing so lifts to one distinguishing $\tau_1$ and $\tau_2$ in $G$ by Lemma~\ref{proxyLifting}, contradicting the efficiency of $Z$.
Therefore, it remains to show that $\eta_1$ and $\eta_2$ are distinguished by $Z$.
For this, we check three cases.

In the first case, both $\tau_1$ and $\tau_2$ lie in the closure of $\Pi(Z)$.
Then $\tau_1$ and $\tau_2$ are distinct ends of $G$ that lie in the closure of $\Pi(Z)$, and so their proxies stem from rays of $\tau_1$ and $\tau_2$ respectively.
Now $Z$ witnesses that these rays are inequivalent in $G$ and, in particular, that they are inequivalent in $G[\Pi]$.
Thus $Z$ distinguishes $\eta_1$ and $\eta_2$ in $H(Z)$ by Lemma~\ref{endSpaceModifiedTorso}.

In the second case, neither $\tau_1$ nor $\tau_2$ lies in the closure of $\Pi(Z)$, and both walk corridors $\gamma_1$ and $\gamma_2$ of~$O(Z)$.
We let $\Xi_1$ and $\Xi_2$ be the separators of the suprema of $\gamma_1$ and $\gamma_2$.
Then, by Proposition~\ref{MagicLemma}, for both $i=1,2$ the induced subgraph $G[\Xi_i\setminus Z]$ is a non-empty clique that is entirely contained in $G[A_i\setminus A_j]$.
Consequently, $Z$ distinguishes $\eta_1$ and $\eta_2$ in $H(Z)$ by Lemma~\ref{endSpaceModifiedTorso}.

In the third case, $\tau_1$ does not lie in the closure of $\Pi(Z)$ and walks a corridor $\gamma_1$ of $O(Z)$ while $\tau_2$ lies in the closure of $\Pi(Z)$.
Then $\tau_2$ must be an end of $G$.
We let $\Xi_1$ be the separator of the supremum of $\gamma_1$.

By Proposition~\ref{MagicLemma} the induced subgraph $G[\Xi_1\setminus Z]$ is a non-empty clique that is entirely contained in $G[A_1\setminus A_2]$.
Since $\eta_1$ stems from the copy of $K^{\aleph_0}$ that is attached to the clique $G[\Xi_1]\subset G[\Pi]$ while $\eta_2$ stems from a ray of $G[\Pi]$ in $\tau_2$, we deduce that $Z$ distinguishes $\eta_1$ and $\eta_2$ in $H(Z)$ by Lemma~\ref{endSpaceModifiedTorso}.
\end{proof}

In the remainder of this subsection we prove Proposition~\ref{MagicLemma}.
For this, we introduce the concept of a pointer.
Basically, the idea is to have a connected subgraph of $G$ that can be employed as an oracle---like we employ rays as oracles for their ends.

\begin{definition}(Pointer)
If $\tau_i$ walks a corridor $\gamma_i$ of $O(Z)$, then a \emph{pointer} of $\tau_i$ is a connected subgraph~$K_i$ of $G[A_{\gamma_i}\setminus\Pi(Z)]\cap G[A_i\setminus A_j]$ that is of the following form.
If $\tau_i$ is an end of $G$, then $K_i$ is a ray in $\tau_i$.
Otherwise $\tau_i$ is an ultrafilter tangle of $G$, and then $K_i$ is a component in $\CC{X_i}$.
\end{definition}

\begin{lemma}\label{PointersExist}
If $\tau_i$ walks a corridor of $O(Z)$, then $\tau_i$ has a pointer.
\end{lemma}
\begin{proof}
If $\tau_i$ is an end, then $\tau_i$ has a ray avoiding $\Pi(Z)$ for $\tau_i$ walks a corridor of $O(Z)$, and as $Z\subset \Pi(Z)$ every such ray is a pointer of $\tau_i$.
Otherwise $\tau_i$ is an ultrafilter tangle.
Then we let $\gamma_i$ be the corridor of $O(Z)$ walked by $\tau_i$.
Let $(C,D)\in\tau_i$ witness that $\tau_i$ walks $\gamma_i$, so $D\setminus C\subset A_{\gamma_i}\setminus\Pi(Z)$.
Using Theorem~\ref{UFtangleDeterminedByFreeUF} we pick $\cC\in U(\tau_i,X_i)$ with $V[\cC]\subset A_i\setminus A_j$ and $\cC'\in U(\tau_i,X_i)$ with $V[\cC']\subset D\setminus C$.
As $U(\tau_i,X_i)$ is a free ultrafilter, the intersection $\cC\cap\cC'\cap\CC{X_i}\in U(\tau_i,X_i)$ is infinite, and every component in this intersection is a pointer of $\tau_i$.
\end{proof}

\begin{lemma}\label{PointersAndDistinguishers}
If the neighbourhood $N(C_i)$ of the component $C_i$ of $G-\Pi(Z)$ containing a pointer $K_i$ of $\tau_i$ is finite, then $N(C_i)$ is the separator of some finite-order separation of $G$ that distinguishes $\tau_1$ and $\tau_2$.
\end{lemma}
\begin{proof}
By the consistency of $\tau_j$ it suffices to find a separation $(A,B)\in\tau_i$ with $(B,A)\le (A_i,A_j)$ and $A\cap B=Y$ where we write $Y=N(C_i)$.
As $K_i$ is a pointer we have $K_i\subset G[A_i\setminus A_j]$.
Since the separator $Z=A_1\cap A_2$ is included in $\Pi(Z)$ we have $C_i\subset G[A_i\setminus A_j]$ as well.
If $\tau_i$ is an end then $\rsep{Y}{C_i}\in\tau_i$ is as desired.
Otherwise $\tau_i$ is an ultrafilter tangle.
If $C_i=K_i$ then $N(C_i)=N(K_i)=X_i$, and employing Theorem~\ref{UFtangleDeterminedByFreeUF} we may pick $\cC\in U(\tau_i,X_i)$ with $V[\cC]\subset A_i$, so $\rsep{X_i}{\cC}\in\tau_i$ is as desired.
Hence we may assume that $C_i\supsetneq K_i$ must meet $X_i$.
Then $C_i=C_Y(X_i)$, and by Lemma~\ref{UFtanglePrincipalGenerator} we have $\rsep{Y}{C_Y(X_i)}\in\tau_i$ as desired.
\end{proof}

\begin{proof}[{Proof of Proposition~\ref{MagicLemma}}]
By Lemma~\ref{PointersExist} we find a pointer $K_i$ of $\tau_i$, and we let $C_i$ be the component of $G-\Pi(Z)$ containing $K_i$.
Then $K_i\subset G[A_{\gamma_i}\setminus\Pi(Z)]$ implies $C_i\subset G[A_{\gamma_i}\setminus\Pi(Z)]$, so we have $N(C_i)\subset\Xi_i$.
First, we show that $C_i$ has a neighbour in $\Xi_i\setminus Z$.
Otherwise $N(C_i)\subset\Xi_i\cap Z$, and then $N(C_i)=Z$ by Lemma~\ref{PointersAndDistinguishers} and the efficiency of $Z$.
Now $Z\subset\Xi_i$ with $Z$ being finite allows us to find a separation $\rsep{X}{\cC}\in\gamma_i$ with $Z\subset X$ contradicting this subsection's assumption on $Z$.
Therefore, $C_i$ has a neighbour in $\Xi_i\setminus Z$.
Next, since $G[\Xi_i]$ is a clique, there is a unique component $D_i$ of $G-Z$ containing $G[\Xi_i\setminus Z]$.
Then $C_i\subset D_i$ as $C_i$ has a neighbour in $\Xi_i\setminus Z$, and so $G[\Xi_i\setminus Z]\subset G[A_i\setminus A_j]$ follows from the pointer $K_i$ being included in $G[A_i\setminus A_j]$.
\end{proof}

\subsection{Proof of the main result}

At last, we prove our main result:

\begin{customthm}{\ref{TreeSetForInfTangles}}
Every connected graph $G$ has a tree set of tame finite-order separations that efficiently distinguishes all its inequivalent $\aleph_0$-tangles.
In particular, equivalent $\aleph_0$-tangles induce the same orientations on the tree set.
\end{customthm}

\begin{proof}
For every modified torso $H$ of $T$ we employ Carmesin's Theorem~\ref{CarmesinEndsTreeSet} to obtain a tree set $T_H$ of $H$-relevant separations that efficiently distinguishes all the ends of $H$.
Then we lift all the separations in all the tree sets $T_H$ and add these to $T$ to obtain an extension $T'$ of $T$.
Then $T'$ is again a tree set by Lemma~\ref{liftCommutesWithProxyOfCorridors}, Corollary~\ref{liftsOfDistinctTorsosAreNested} and Lemma~\ref{liftRespectsNested}.

First, we show that $T'$ efficiently distinguishes every two inequivalent $\aleph_0$-tangles of $G$.
For this, let $\tau_1$ and $\tau_2$ be two inequivalent $\aleph_0$-tangles of $G$.
We have to find a separation in $T'$ that efficiently distinguishes $\tau_1$ and $\tau_2$.
Pick some finite-order separation of $G$ with separator $Z$ say that efficiently distinguishes $\tau_1$ and $\tau_2$.
If $Z$ is contained in some critical vertex set of $G$, then by Lemma~\ref{CaseOne} we find a separation in $T\subset T'$ that efficiently distinguishes $\tau_1$ and~$\tau_2$.
Otherwise $Z$ is not contained in any critical vertex set of~$G$.
However, $Z$ is generous by Lemma~\ref{RelevantImpliesGenerous}, and so by Lemma~\ref{GenerousIsPrincipalForPrincipals} induces a consistent orientation of $T$ whose part contains~$Z$.
Then by Proposition~\ref{distinctCorridorsDistinctProxies} the $\aleph_0$-tangles $\tau_1$ and $\tau_2$ have distinct proxies $\eta_1$ and~$\eta_2$ in the modified torso $H$ of that orientation, and $Z$ efficiently distinguishes $\eta_1$ and $\eta_2$ in~$H$.
Thus there is a separation in $T_H$ of order $\vert Z\vert$ that distinguishes $\eta_1$ and~$\eta_2$.
By Lemma~\ref{proxyLifting} this separation lifts to a separation of $G$ that distinguishes $\tau_1$ and~$\tau_2$.
This lift still has order $\vert Z\vert$ and lies in $T'$.

Second we show that all separations in $T'$ are tame.
Every separation in $T$ is tame.
And by Lemma~\ref{LiftsAreTame} the lifts of all $T_H$ are tame as well.
\end{proof}

\section{Appendix}\label{sec:appendix}

\subsection{Compactifications}

A \emph{compactification} of a topological space $X$ is an ordered pair $(K,h)$ where $K$ is a compact topological space and $h\colon X\hookrightarrow K$ is an embedding of $X$ as a dense subset of $K$.
Sometimes we also refer to $K$ as a \comp\ of $X$ if the embedding $h$ is clearly understood (e.g. if $h$ is the identity on $X$).
The space $K\setminus h[X]$ is called the \emph{remainder} of the \comp .

Suppose now that $X$ is a discrete topological space.
Since $X$ is locally compact\footnote{A topological space $X$ is \emph{locally compact} if for each of its points there is some compact subspace of $X$ which includes an open neighbourhood of that point.}, $X$ is open in all of its \HDcomp s 
(cf.~\cite[Theorem~3.6.6]{EngelkingBook}).
If we pair the space $\beta X$ of all ultrafilters on $X$ carrying the topology whose basic open sets are of the form $\{U\in\beta X\mid A\in U\}$, one for each $A\subseteq X$, with the embedding that sends every $x\in X$ to the principal ultrafilter on $X$ generated by $\{x\}$, then this yields the Stone-Čech \comp\ of~$X$.
By the \SC\ property, every continuous function $f\colon X\to T$ into a compact \HD\ space $T$ has a unique continuous extension $\beta f\colon\beta X\to T$ with $\beta f\rest X=f$ (cf.~\cite[Theorem~3.5.1]{EngelkingBook}).
\begin{theorem}[{\cite[Corollary 7.4]{TheoryOfUltrafilters}}]
\label{Top:compactification:StoneCechCard}
If $X$ is an infinite set, then $\vert\beta X\vert = 2^{2^{\vert X\vert}}$.
\end{theorem}

\subsection{Inverse limits of inverse systems}

A partially ordered set $(I,\le)$ is said to be \emph{directed} if for every two $i,j\in I$ there is some $k\in I$ with $k\ge i,j$.
Let $(\,X_i\mid i\in I\,)$ be a family of topological spaces indexed by some directed poset $(I,\le)$.
Furthermore, suppose that we have a family $(\,\varphi_{ji}\colon X_j\to X_i\,)_{i\le j\in I}$ of continuous mappings which are the identity on $X_i$ in case of $i=j$ and which are \emph{compatible} in that $\varphi_{ki}=\varphi_{ji}\circ\varphi_{kj}$ for all $i\le j\le k$. Then both families together are said to form an \emph{inverse system}, and the maps $\varphi_{ji}$ are called its \emph{bonding maps}. 
We denote such an inverse system by $\{X_i,\varphi_{ji},I\}$ or $\{X_i,\varphi_{ji}\}$ for short if $I$ is clear from context.
Its \emph{inverse limit} $\invLim{} X_i=\invLim{}(\,X_i\mid i\in I\,)$ is the topological space
\begin{align*}
\invLim{} X_i=\{\,(x_i)_{i\in I}\mid \varphi_{ji}(x_j)=x_i\text{ for all }i\le j\,\}\subseteq \prod_{i\in I}X_i.
\end{align*}
Whenever we define an inverse system without specifying a topology for the spaces $X_i$ first, we tacitly assume them to carry the discrete topology.
If each $X_i$ is (non-empty) compact \HD , then so is the inverse limit~$\invLim{}X_i$.

\bibliographystyle{amsplain}
\bibliography{BIB_11_18,localbib}

\end{document}